\documentclass[11pt, reqno]{amsart}
\usepackage{amsmath, amssymb, amsthm, graphicx,marvosym}
\usepackage{cancel}
\usepackage[nobysame]{amsrefs}[11pts]

\usepackage{color}

\usepackage{tikz-cd}
\usepackage{tikz}
\usepackage{caption}
\usepackage{amsmath}
\usepackage{amssymb}
\usepackage{amsthm}
\usepackage{enumerate}
\newtheorem{thm}{Theorem}
\newtheorem{lemma}{Lemma}[section]
\newtheorem{cor}[lemma]{Corollary}
\newtheorem{prop}[lemma]{Proposition}
\newtheorem{definition}[lemma]{Definition}
\newtheorem{remark}[lemma]{Remark}
\newcommand \nc{\newcommand}
\newcommand{\ben}{\begin{eqnarray}}
\newcommand{\een}{\end{eqnarray}}
\newcommand{\beno}{\begin{eqnarray*}}
\newcommand{\eeno}{\end{eqnarray*}}

\makeatletter \@addtoreset{equation}{section} \makeatother

\nc{\ba}{\begin{array}}\nc{\ea}{\end{array}}

\nc{\be}{\begin{eqnarray}}\nc{\ee}{\end{eqnarray}}
\nc{\beq}{\begin{equation}}\nc{\eeq}{\end{equation}}
\nc{\bex}{\begin{eqnarray*}}\nc{\eex}{\end{eqnarray*}}
\nc{\btm}{\begin{theorem}} \nc{\etm}{\end{theorem}}
\nc{\blm}{\begin{lemma}} \nc{\elm}{\end{lemma}}
\nc{\va}{\varphi}
\nc{\ve}{\varepsilon}

\def\ds{\displaystyle}

 \newcommand{\n}{{\bf n}}
\newcommand{\bu}{{\bf u}}

\newcommand{\R}{\mathbb {R}}

\def\({\left(\begin{array}{cccccc}}
\def\){\end{array}\right)}

\def\bes{\begin{eqnarray}}
\def\ees{\end{eqnarray}}

\title[Wave model for liquid-crystals]{The Poiseuille flow of the full Ericksen-Leslie model for nematic liquid crystals: The general Case}

\author{Geng Chen}
\address[G. Chen] {Department of Mathematics, University of Kansas, Lawrence, KS 66045, U.S.A.,}
\email{\tt gengchen.edu}

\author{Weishi Liu}
\address[W. Liu]{Department of Mathematics,
University of Kansas, Lawrence, KS 66045, U.S.A.}
\email{\tt wsliu@ku.edu}

\author{Majed Sofiani}
\address[M. Sofiani]{Department of Mathematics, University of Kansas, Lawrence, KS 66045, U.S.A.}
\email{\tt sofiani@ku.edu}

\date{\today}

\begin{document}

\begin{abstract} In this work, we study the Cauchy problem of Poiseuille flow of the full Ericksen-Leslie model for nematic liquid crystals. The model is a coupled system of two partial differential equations (PDEs): One is a quasi-linear wave equation for the director field representing the crystallization of the nematics, and the other is  a parabolic PDE for the velocity field characterizing the liquidity of the material. We extend the work in [Chen, et. al. {\em Arch. Ration. Mech. Anal.} {\bf 236} (2020), 839-891] for a special case to the general physical setup. The Cauchy problem  is shown to have global solutions beyond singularity formation. Among a number of progresses made in this paper, a particular contribution is a systematic treatment of a parabolic PDE with only H\"older continuous diffusion coefficient and rough (maybe unbounded) nonhomogeneous terms.
\end{abstract}

\maketitle

\section{Introduction}\label{intro}

In this paper we consider existence and regularity for Poiseuille flow of nematic liquid crystals. Liquid crystals have many forms and a particular form is the nematic whose molecules can be viewed as 
rod-like/thread-like.  Macroscopically,  the state of a nematic liquid crystal is  characterized by its velocity field ${\bf u}$ for the flow and its director field  ${\bf n}\in \mathbb S^2$ for the alignment of the rod-like feature.  These two characteristics interact with each other so that any distortion of the director $\n$ causes 
a motion ${\bf u}$  and, likewise,  any flow ${\bf u}$ affects the alignment $\n$. Using the convention to denote $\dot{f}=f_t+\bu\cdot \nabla f$ the material derivative, the full Ericksen-Leslie model for nematics  is given as follows
\begin{equation}\label{wlce}
\begin{cases}
\rho\dot \bu+\nabla P=\nabla\cdot\sigma-\nabla\cdot\left(\frac{\partial W}{\partial\nabla \n}\otimes\nabla \n\right),
\\
\nabla\cdot \bu=0,\\ 
\nu{\ddot \n}=\lambda\n-\frac{\partial W}{\partial  \n}-{\bf g}+\nabla\cdot\left(\frac{\partial W}{\partial\nabla \n}\right), \\
|\n|=1.\\
\end{cases}
\end{equation}
In  (\ref{wlce}),  $P$ is the pressure, $\lambda$ is the Lagrangian multiplier of the constraint $|\n|=1$, $\rho$ is the density, $\nu$ is the inertial coefficient of the director $\n$, and $W$, 
${\bf g}$  and $\sigma$ are the Oseen-Frank energy,  the kinematic transport and the viscous stress tensor, respectively (see, e.g., \cites{DeGP,ericksen62, frank58, liucalderer00,leslie68,Les, lin89} for details).

\subsection{Poiseulle flow of nematics} For Poiseulle flows of  nematics with a choice of coordinates system, ${\bf u}$ and ${\bf n}$ take the form
\[{\bf u}(x,t)= (0,0, u(x,t))^T \;\mbox{ and }\;{\bf n}(x,t)=(\sin\theta(x,t),0,\cos\theta(x,t))^T,\]
where the motion ${\bf u}$ is along the $z$-axis and   the director ${\bf n}$ lies in the $xz$-plane with angle $\theta$ made from the $z$-axis.   For this case of Poiseulle flows, taking $\rho=\nu=1$ for simplicity, the Ericksen-Leslie  model is reduced to system (\ref{sysf}) and (\ref{sysw}) below, whose  derivation is available in the literature (see, e.g., \cite{CHL20}),
\begin{align}\label{sysf}
    \ds u_t&=\left(g(\theta)u_x+h(\theta)\theta_t\right)_x,\\\label{sysw}
\theta_{tt}+\gamma_1\theta_t&=c(\theta)\big(c(\theta)\theta_x\big)_x-h(\theta)u_x.
\end{align}
We will be interested in Cauchy problem of the Poiseulle flow  
 with initial data 
\beq\label{initial}
u(x,0)=u_0(x)\in H^1(\R),\; \theta(x,0)=\theta_0(x)\in H^1,
\; \theta_t(x,0)=\theta_1(x)\in L^2(\R).
\eeq
We further impose an assumption
\beq\label{tecinitial}\begin{split}
u'_0(x)+\frac{h(\theta_0(x))}{g(\theta_0(x))}\theta_1(x):=J_0&\in H^1(x) \cap L^1(\R).
\end{split}
\eeq
To this end, we make an important comment on condition (\ref{tecinitial}). It is known singularity may form in finite time in the sense that $u_x$ and $\theta_t$ can approach as $(x,t)\to (x^*,t^*)$ for some $(x^*,t^*)$. The quantity $J:=u_x+\frac{h(\theta)}{g(\theta)}\theta_t$ turns out to be in $C^\alpha\cap L^2\cap L^\infty(\R\times [0,T])$ with $0\leq \alpha<\frac{1}{4}$ for any $t$ when condition (\ref{tecinitial}) holds. The better regularity of $J$  reveals a cancellation of finite time singularities of $u$ and $\theta$ and plays crucial role for the study of existence after singularity formation. This key observation was first made in \cite{CHL20} for a special case.
 We further need $J_0\in L^1(\R)$ such that an important auxiliary function $A_0(x)=\int_{-\infty}^x J_0(z)dz$ is well defined. The condition $J_0\in H^1$ is needed for estimates in (\ref{Gfbound}) below and is roughly in line with the space of $J$. 

 
 The functions $c(\theta)$, $g(\theta)$ and $h(\theta)$ in the model  are given by  
\begin{align}\label{fgh}\begin{split}
 g(\theta):=&\alpha_1\sin^2\theta\cos^2\theta+\frac{\alpha_5-\alpha_2}{2}\sin^2\theta+\frac{\alpha_3+\alpha_6}{2}\cos^2\theta+\frac{\alpha_4}{2},\\
  h(\theta):=&\alpha_3\cos^2\theta-\alpha_2\sin^2\theta=\frac{\gamma_1+\gamma_2\cos(2\theta) }{2},\\
  c^2(\theta):=&K_1\cos^2\theta+K_3\sin^2\theta,
 \end{split}
 \end{align}
 where, depending on the nematic material and temperature, $K_1>0$ and $K_3>0$ are Frank's coefficients for splay and bending energies,  and $\alpha_j$'s $(1 \leq j \leq 6)$ are the Leslie dynamic coefficients. The following relations are assumed in the literature.
\begin{align}\label{a2g}
\gamma_1 =\alpha_3-\alpha_2,\quad \gamma_2 =\alpha_6 -\alpha_5,\quad \alpha_2+ \alpha_3 =\alpha_6-\alpha_5. 
\end{align}
The first two relations are compatibility conditions, while the third relation is called a Parodi's relation, derived from Onsager reciprocal relations expressing the equality of certain relations between flows and forces in thermodynamic systems out of equilibrium (cf. \cite{Parodi70}). 
They also satisfy the following empirical relations (p.13, \cite{Les}) 
\begin{align}\label{alphas}
&\alpha_4>0,\quad 2\alpha_1+3\alpha_4+2\alpha_5+2\alpha_6>0,\quad \gamma_1=\alpha_3-\alpha_2>0,\\
&  2\alpha_4+\alpha_5+\alpha_6>0,\quad 4\gamma_1(2\alpha_4+\alpha_5+\alpha_6)>(\alpha_2+\alpha_3+\gamma_2)^2\notag.
\end{align}
Note that the fourth relation  is implied by the third together with the last relation and the last can be rewritten as
$\gamma_1(2\alpha_4+\alpha_5+\alpha_6)> \gamma_2^2$.
 {Very importantly, relations (\ref{a2g}) and (\ref{alphas}) imply that (see formula (2.4) in \cite{CHL20}),  for some constant $ \overline C>0$,
\begin{align}\label{positiveDamping}
g(\theta)\ge g(\theta)-\frac{h^2(\theta)}{\gamma_1}>\overline C.
\end{align}
}

It is known that the solution of system \eqref{sysf} and \eqref{sysw} generically form finite time cusp singularity \cites{CS22,CHL20}.{ The goal of this paper is to establish a global existence of H\"older continuous solutions for the Cauchy problem  \eqref{sysf}-\eqref{tecinitial} beyond singularity formation as stated in Theorem \ref{main} below.}

\subsection{Directly relevant results and our results.}
 
 In \cite{CHL20}, a special case of Poiseuille flow was treated.  More precisely, the authors chose the parameters as
\[\rho=\nu=1,\; \alpha_1=\alpha_5=\alpha_6=0,\;\alpha_2=-1,\;\alpha_3=\alpha_4=1,\]
which result in $\gamma_1=2$, $\gamma_2=0$, and $g=h=1$. With this special choice of parameters, system (\ref{sysf}) and (\ref{sysw}) becomes
\begin{align}\label{simeqn0}
\begin{split}
 u_t=&(u_x+\theta_t)_x,\\
\theta_{tt}+2\theta_t=&c(\theta)(c(\theta)\theta_{x})_x
 -u_x.
\end{split}
\end{align}
 In \cite{CHL20},  on one hand, the authors constructed solutions for \eqref{simeqn0} with smooth initial data that produce, in finite time, cusp singularities—blowups of $u_x$ and $\theta_t$. The method directly extends  that of \cites{CZ12, GHZ} for  variational wave equations.   On the other hand, the global existence of weak solutions, which are H\"older continuous, of system \eqref{simeqn0} were established for general initial data similar to (\ref{initial}). 
 The latter resolved satisfactorily the physical concerns from application point of view about what happens after the finite time singularity formation. 
  A crucial ingredient for existence beyond singularity formation is the  identification of the quantity
\[J(x,t)=u_x(x,t)+\theta_t(x,t)\]
 and the reveal of a singularity cancellation—the quantity $J$ remains bounded and H\"older continuous while  its components $u_x$ and $\theta_t$ approach infinity at formations of cusp singularities.


%
%

The change of coordinates framework in \cite{BZ}  for the variational wave equations was used to cope with the wave part $\eqref{simeqn0}_2$, and will be used in this paper too for \eqref{sysw}.
The detailed idea will be given in Section \ref{idea}. See other works on the global well-posedness of H\"older continuous solutions for variational wave equations \cites{BC,BC2015,BH,BHY,BCZ, CCD, ZZ03,ZZ10,ZZ11,GHZ}.

 In a recent paper \cite{CS22}, the singularity formation for the general system \eqref{sysf} and \eqref{sysw} was established. 
 As mentioned above, we are concerning with the global existence of the Cauchy problem for the general system \eqref{sysf} and \eqref{sysw}. {It should be pointed out that the generalization is far beyond straightforward. 
One apparent trouble  is that the diffusion coefficient $g(\theta(x,t))$ in the parabolic equation \eqref{sysf} is only H\"older continuous, which creates difficulties in handling the quantity $J$ for the singularity (see system (\ref{J_eq}) and the discussion followed for details). This leads the authors to introduce and work with the potential $A$ of $J$ in (\ref{VarA}).  Another difficulty is caused by rough (worse than H\"older continuous) non-homogeneous terms in the parabolic equation for $A$, in addition to the diffusion coefficient $g(\theta(x,t))$ being only H\"older continuous. This difficulty is overcome with a careful analysis in \cite{MS} that relies on but goes beyond treatments in \cite{Fri}. The work in \cite{MS} has a much more broad interest besides a direct application to the present work.}

For the statement of our result, we need the following definition of weak solutions.
\begin{definition}\label{def1}
For any fixed time $T\in (0,\infty)$, we say that $(u(x,t),\theta(x,t))$ is a weak solution of \eqref{sysf}, \eqref{sysw}, and \eqref{initial} over $\R\times [0,T]$ if
\begin{itemize}
 \item 
 For any test function $\phi\in H^1_0(\R\times (0,T)),$
\begin{align}\label{thetaweak1}
    \int_0^T\int_\R\theta_t\phi_t-\gamma_1\theta_t\phi\,dx\,dt=\int_0^T\int_\R (c(\theta)\theta_x)(c(\theta)\phi)_x+hu_x\phi\,dx\,dt,
\end{align}
\begin{align}\label{uweak}
    \int_0^T\int_R u\phi_t-(gu_x+h\theta_t)\phi_x\,dx\,dt=0.
\end{align}

\item The initial data
\begin{align}\label{data}
    u(x,0)=u_0(x),\; \theta(x,0)=\theta_0(x),\; \text{and}\; \theta_t(x,0)=\theta_1(x)
\end{align}
hold point-wise for $u(x,0)$ and $\theta(x,0),$ and in $L^p$ sense for $\theta_t(x,0)$ for any $p\in[1,2).$ 
\end{itemize}
\end{definition}

 Throughout the paper, we will always assume relations \eqref{a2g} and \eqref{alphas} for the Leslie coefficients $\alpha_j$'s and refer to $g$, $h$ and $c$ as the functions given in \eqref{fgh}. 
\begin{thm}[Global Existence]\label{main}  For any fixed time $T\in (0,\infty)$, the Cauchy problem of system \eqref{sysf} and \eqref{sysw} with  the initial conditions $u_0(x)$, $\theta_0(x)$ and $\theta_1(x)$ given in \eqref{initial} and \eqref{tecinitial} has a weak solution $(u(x,t),\theta(x,t))$ defined on $\R\times [0,T]$ in the sense of Definition \ref{def1}. 
    Moreover,
    \beq\label{u_est}
    u(x,t)\in L^\infty([0,T],H^1_{loc}(\R))\cap L^2([0,T],H^1(\R))\cap L^\infty(\R\times [0,T])\eeq and 
    \[\theta(x,t)\in C^{1/2}(\R\times [0,T])
    \]
and, for any $t\in[0,T]$, the associated energy
\begin{align}\label{calEdef}
\mathcal{E}(t):=\int_\R\theta_t^2+c^2(\theta)\theta_x^2+u^2\,dx
\end{align}
    satisfies 
    \begin{align}
        \mathcal{E}(t)\leq \mathcal{E}(0)-\int_0^t\int_\R (u_x+\frac{h}{g}\theta_t)^2+\theta_t^2\,dx\,dt.
    \end{align}
\end{thm}

\medskip

The rest of this paper is organized as follows. In Section \ref{idea}, we introduce the main idea of this work. In Section \ref{J2theta}, as the first step in carrying out the main idea, we analyze the wave equation for $\theta$ with a prescribed forcing term.
 In Section \ref{HK}, we recall some basic properties from \cite{Fri} on parabolic differential operators with only H\"older continuous diffusion coefficients. In Sections \ref{J2v}, we apply the formulation in Section \ref{HK} and results in \cite{MS} to analyze $u$-component.  In Sections \ref{FixedPoint}-\ref{bdE} we prove the  existence of weak solution for the Cauchy problem \eqref{sysf} and \eqref{sysw}.

\section{Main idea of this work}\label{idea}
 The approach developed in \cite{CHL20} for the special case with $g=h=1$ provides a framework for the general system \eqref{sysf} and \eqref{sysw}. There are, however, a number of crucial issues in this generalization.

Note that singularity formation is unavoidable in general \cite{CS22}. Similar to \cite{CHL20}, we introduce a new variable 
\[v=\int_{-\infty}^x u\,dx\]
 and obtain, from \eqref{sysf},
\beq\label{veq}
v_t=g(\theta)v_{xx}+h(\theta) \theta_t=g(\theta) u_x+h(\theta) \theta_t.
\eeq
Motivated by singularity cancellation revealed in \cite{CHL20},  we also introduce
\begin{align}\label{VarJ}
J=\frac{v_t}{g(\theta)}=u_x+\frac{h(\theta)}{g(\theta)}\theta_t,
\end{align}
which agrees with the function $J$  in \cite{CHL20} for the special case with $g=h=1$. 

Let's explain why $J$ enjoys an enhanced regularity for the simplified model considered in \cite{CHL20}:
\begin{align}\label{simeqn_20} \begin{split}
 v_t=&v_{xx}+\theta_t,\\ 
\theta_{tt}+\theta_t=&c(\theta)(c(\theta)\theta_{x})_x
-v_t,
\end{split}
\end{align}
where $J(x,t)=v_t=u_x+\theta_t$ for this simplified model. In \cite{CHL20}, we can show that $J$ has finite $L^2$, $L^\infty$ and $C^\alpha$ norms, with $\alpha\in(0, 1/4)$.  Given the fact that both $u_x$ and $\theta_t$ may blow up, the bound and regularity of $J$ are fundamentally important.
This result holds because of the different ``scales'' of time variable $t$ in heat equation and in wave equation. Very roughly speaking, $J=v_t\approx H*\theta_{tt}$ for the heat kernel $H$, which is very bad. However, one can use the wave equation to switch $\partial_{tt}$ to $\partial_{xx}$ and other lower order terms, although the switch is highly non-trivial and needs to be done for the weak solution. Since $H_x$ has much better regularity than $H_t$, we can show the enhanced regularity for $J$.

It follows from \eqref{veq} and \eqref{VarJ} that $J$ satisfies
\beq\label{J_eq}
(g(\theta)J)_t=g(\theta) (g(\theta)J)_{xx}+h(\theta) \theta_{tt}+h'(\theta) \theta_t^2+g'(\theta)\theta_t
J-g'(\theta)\frac{h(\theta)}{g(\theta)}\theta_t^2.\eeq
In \cite{CHL20}, we also take advantage of the explicit expression of $J$ in terms of $\theta$, using the heat kernel $H$, when we do the Schauder estimate.
However, when we come to the general system \eqref{sysf}-\eqref{sysw}, it turns out the coefficient $g(\theta)$ in $(g(\theta) J)_{xx}$ is  only H\"older continuous in general. 
And for the general case, we do not have an explicit formula using the heat kernel as in the special case.
The work in \cite{Fri} provides an implicit  expression for the kernel which can be used to treat the nonhomogeneous terms in (\ref{J_eq}) in an indirect way (see Section \ref{HK} for more details). In order to follow the framework in \cite{Fri} for the nonhomogeneous parabolic equation \eqref{J_eq},  
we introduce a new auxiliary function $A,$

\begin{align}\label{VarA}
A(x,t)=\hat A(x,t)-A_0(x)=\int_{-\infty}^xJ(z,t)\,dz-\int_{-\infty}^xJ(z,0)\,dz,
\end{align}
that is $A_x=J-J_0,$ with $J_0=J(x,0)$ determined by \eqref{tecinitial}. 
Essentially, $A$ is the potential of the function $J$. By considering  $A$, we can overcome some difficulty when we differentiate the non-explicit kernel of the heat equation with H\"older continuous coefficient \cite{MS}.

 By comparing equation (\ref{J_eq}) with equation (\ref{AAeq}) for $A$ given below, we avoid some complication in handling the rough term $\theta_{xx}$.

Instead of working directly on system \eqref{sysf} and \eqref{sysw}, we will treat the equivalent system in terms of the quantities $(v, \theta,A)$ as
\begin{align}
  v_t=g(\theta)v_{xx}&+h(\theta) \theta_t,\label{ueq}\\
   \theta_{tt}+\big(\gamma_1-\frac{h^2(\theta)}{g(\theta)}\big)\theta_t&=c(\theta)(c(\theta)\theta_x)_x-h(\theta)\hat A_x,\label{thetaeq}\\
    A_t=g(\theta)A_{xx}-&\gamma_1A+g'(\theta)\theta_xA_x+g'(\theta)\theta_xJ_0+F(\theta, v), \label{AAeq}
   \end{align}
  where
\begin{align}\label{Fun}\begin{split}
F&=G+f,\\
f&=[\gamma_1-\frac{h^2}{g}]v_x+\frac{h(\theta)c^2(\theta)}{g}\theta_x+g(\theta)J'_0,\\
G&=\int_{-\infty}^x[\frac{h'}{g}-\frac{g'h}{g^2}]\theta_t^2-[\gamma_1-\frac{h^2}{g}]'\theta_zv_z-(\frac{h(\theta)c(\theta)}{g})'c(\theta)\theta_z^2\,dz-\gamma_1 A_0,
    \end{split}  
\end{align}
and
\beq\label{a0a0x}
A_0(x)=\int_{-\infty}^xJ_0(z)\,dz=\int_{-\infty}^xJ(z,0)\,dz,
\eeq  
Note $A(x,0)=0.$
 
  The splitting of $F=f+G$ is based on different regularities of each term, as we will see later.
  A derivation for (\ref{AAeq}) from  \eqref{J_eq} is provided in Appendix \ref{App-A}.\\
   
Roughly, Theorem \ref{main} will be proved in the following steps.
 \bigskip


\paragraph{\bf Step 1:} For any given $J(x,t)\in C^\alpha\cap L^2\cap L^\infty$ for some $\alpha>0$, we consider the wave equation (\ref{thetaeq}) with $A_x$ being replaced by $J$ 
\begin{align}\label{waveJ}
    \theta_{tt}+\big(\gamma_1-\frac{h^2(\theta)}{g(\theta)}\big)\theta_t&=c(\theta)(c(\theta)\theta_x)_x-h(\theta)J.
    \end{align}\\
    Using very similar method as in \cite{CHL20},  the existence of a $C^{1/2}$ solution $\theta^J$ of (\ref{waveJ}) will be shown in Section \ref{J2theta}.
     \bigskip
    
\paragraph{\bf Step 2:}   With $\theta^J$ obtained from Step 1,  we then solve $v^J$  from the equation
    \begin{align}\label{vJ}
    v_t=g(\theta^J)v_{xx}+h(\theta^J)\theta_t^J,
    \end{align}
 and show that both $v^J$ and $u^J=v_x^J$ are in $ C^\alpha\cap L^2\cap L^\infty$ in Section \ref{J2v}.
    \bigskip
    
\paragraph{\bf Step 3:} With $(v^J,\theta^J)$ from the above steps, we will solve for $A^J$ from
\[A_t=g(\theta^J)A_{xx}-\gamma_1A+g'\theta_x^JJ+F(\theta^J,v^J).\]
An expression of $A^J=\mathcal{N}(J)$ by the so called parametrix method in \cite{Fri} is very helpful. Recall that $A_x+J_0=J$. After setting $\mathcal{M}(J)=\left(\mathcal{N}(J)\right)_x+J_0$, we then have  a fixed point problem for the map $J\mapsto \mathcal{M}(J)$ that will be analyzed by the Schauder fixed point theory in Section \ref{FixedPoint}.

\begin{center}
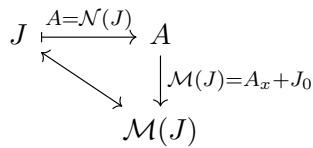
\begin{figure}[h]
\begin{tikzcd}
J  \arrow[dr,leftrightarrow] \arrow[r, mapsto,"A=\mathcal{N}(J)"] & A \arrow[d, "\mathcal{M}(J)=A_x+J_0"] \\
& \mathcal{M}(J)
\end{tikzcd}
\caption{A diagram explaining the relation between the quantities $A,J$ and the map $\mathcal{M}.$}
\label{diagram}
\end{figure}
\end{center}

\section{The solution $\theta$ of (\ref{waveJ}) with a fixed  $J$}\label{J2theta}

Given any function $J(x,t)\in L^2\cap L^\infty\cap C^\alpha$ for some $\alpha>0,$ we consider (\ref{waveJ}) recast below
\begin{align}\label{waveAx}
    \theta_{tt}+\big[\gamma_1-\frac{h^2(\theta)}{g(\theta)}\big]\theta_t&=c(\theta)(c(\theta)\theta_x)_x-h(\theta)J.
\end{align}
It is crucial that   $\gamma_1-\frac{h^2(\theta)}{g(\theta)}>C_*$  for some constant $C_*>0$ (see (\ref{positiveDamping})).


Using the change of coordinates method in \cites{BZ,CHL20}, we can prove the global existence of weak solution for \eqref{waveAx} in a very similar way as for the simplified system. To make this paper self-contained, we include the proof.\\

Given a point $(x_0,t_0)$, we define the characteristic curves
    \[x^\pm(s)\equiv x^\pm(s;\, x_0,\, t_0)\]
    by the solutions $x^\pm(s)$ of
    \[\frac{d x^\pm(s)}{d s}=\pm c\bigl(\theta(x^\pm(s),s)\bigr)\;\mbox{ with }\; x^\pm(t_0)=x_0.\]
 Note that $x^\pm(0)$ are the intersections of the characteristic curves $x^\pm(s)$ with the $x$-axis.    
      With the help of the following gradient variables
      \begin{align}\label{RS}\begin{split}
          S(x,t)= &\theta_t(x,t) - c(\theta(x,t)) \theta_x(x,t),\\
 R(x,t)=& \theta_t(x,t)+c(\theta(x,t)) \theta_x(x,t),
\end{split}
      \end{align}
we make the change of coordinates $(x,t)\mapsto (X,Y)$:
\begin{align}\label{charaXY}\begin{split}
X\equiv X(x,t):=&	\int_1^{x^-(0;\, x,\, t)} (1+ R^2(x',0))\,d x',\\
 Y\equiv Y(x,t):=&	\int_{x^+(0;\, x,\, t)}^1 (1+ S^2(x',0))\,d x'.
 \end{split}
\end{align}
 Note that, for any differentiable function $f$, one has
\begin{align}\label{forward}\begin{split}
f_t(X,Y) +c(\theta) f_x(X,Y)&=2c X_x f_X,\\ 
f_t(X,Y)-c(\theta) f_x(X,Y)&=-2c Y_x f_Y.
\end{split}
\end{align}

In order to complete the system, we introduce several variables: 
\beq\label{defwz}
w:=2\arctan R,\quad z:=2\arctan S,
\eeq
and 
\beq\label{defpq}
p:=\frac{1+R^2}{X_x},\quad q:=\frac{1+S^2}{-Y_x}.
\eeq


We write the system with respect to the new coordinates to obtain
\begin{align}\label{semisys11}
&\theta_X=\frac{\sin w}{4c(\theta)}p,\quad \theta_Y=\frac{\sin z}{4c(\theta)}q
\end{align}
and to close the system we derive the equations for $z,\omega,p$ and $q$
\begin{align}
&\textstyle z_X=p
\Big\{\frac{c'}{4c^2}(\cos^2\frac{w}{2}-\cos^2\frac{z}{2})+\frac{b(\theta)}{4c}(\sin w\cos^2\frac{z}{2}+\sin z\cos^2\frac{w}{2})-\frac{h(\theta)}{c}J\cos^2\frac{z}{2}\cos^2\frac{w}{2} \Big\},\\
&\textstyle w_Y=q
\Big\{\frac{c'}{4c^2}(\cos^2\frac{z}{2}-\cos^2\frac{w}{2})+\frac{b(\theta)}{4c}(\sin w\cos^2\frac{z}{2}+\sin z\cos^2\frac{w}{2})-\frac{h(\theta)}{c}J\cos^2\frac{z}{2}\cos^2\frac{w}{2}\Big\}, \\
&\textstyle p_Y=pq\Big\{\frac{c'}{8c^2}(\sin z-\sin w)
+\frac{b(\theta)}{2c}(\frac{1}{4}\sin w\sin z+\sin^2 \frac{w}{2}\cos^2\frac{z}{2})
-\frac{h(\theta)}{2c}J\sin w\cos^2\frac{z}{2}\Big\},\\\label{semisys55}
&\textstyle q_X=pq\Big\{\frac{c'}{8c^2}(\sin w-\sin z)
+\frac{b(\theta)}{2c}(\frac{1}{4}\sin w\sin z+\sin^2 \frac{z}{2}\cos^2\frac{w}{2})
-\frac{h(\theta)}{2c}J\sin z\cos^2\frac{w}{2}\Big\},
\end{align}

and

\beq
\left\{
\begin{array}{ll}\label{eqnxt}
\ds x_X=\frac{1}{2X_x}=\frac{1+\cos w}{4}p,\\
\\
\ds x_Y=\frac{1}{2Y_x}=-\frac{1+\cos z}{4} q,
\end{array}
\right.
\qquad 
\left\{
\begin{array}{ll}
\ds t_X=\frac{1}{2cX_t}=\frac{1+\cos w}{4c}p,\\
\\
\ds t_Y=-\frac{1}{2cY_t}=\frac{1+\cos z}{4c} q,
\end{array}
\right.
\eeq
where $b(\theta)=\frac{h^2(\theta)}{g(\theta)}-\gamma_1.$ See \cite{CHL20} for the derivation of \eqref{semisys11}-\eqref{eqnxt}.\\

Comparing this system with that in \cite{CHL20} for the special case $g=h=1$ and $\gamma_1=2$, one can see that the appearance of the bounded and smooth functions $g$ and $h$ does not create any difficulties applying the process in \cite{CHL20}.

\begin{prop}[Local Existence]\label{localwave}
There exists $ T>0$ sufficiently small such that system \eqref{semisys11}-\eqref{semisys55} has a solution $(\theta,z,\omega, p,q)(X,Y)$ defined on
\[\Omega_ T:=\{(X,Y)\in\R^2: d\big((X,Y),\gamma\big)\leq T\},\] where $\gamma$ is the curve in the $(X,Y)$ plane corresponding to the line $t=0$ on the $(x,t)$ plane and $d(\cdot,\gamma)$ is the distance between the curve and a point.
\end{prop}
\begin{proof}
The proof is similar to that in \cite{CHL20} and is outlined in Appendix \ref{app-semilinear}.
\end{proof}

Now, to extend the solution globally, meaning to any arbitrary time $0<T<\infty,$ we need some a prior uniform bound on $p$ and $q$.

\begin{lemma}\label{lemma3.3}Consider any solution of \eqref{semisys11}-\eqref{semisys55} constructed in the local existence result with $t\in [0, T]$. Then, we have 
\beq\label{bdspq}
0<A_1\leq \max_{(X,Y)\in \Omega_ T }\left\{p(X,Y),\, q(X,Y)\right\}\leq A_2,
\eeq
for some constants $A_1$ and $A_2$ independent of $ T.$
\end{lemma}
\begin{proof}
We skip the proof since it is entirely similar to that for Lemma 6.2 in \cite{CHL20}.
\end{proof}
Now we need to transfer the solution back to the original coordinate system $(x,t)$ using \eqref{eqnxt}. We note that in general, due to the lack of enough regularity of the solution $\theta,$ we might well lose the  uniqueness of the characteristic curves. So $x(X,Y),t(X,Y)$ might not be well defined. Instead, we will show that $\theta(x,t)$ is always well defined. In fact, a possible scenario in this case is that we might have the following
    \[X_1:=X(x_1,t_1)=X(x_2,t_2)=:X_2\] and \[Y_1:=Y(x_1,t_1)=Y(x_2,t_2)=:Y_2.\] But, one can show that $\theta(X_1,Y_1)=\theta(X_2,Y_2)$ and hence $\theta$ remains well defined (\cite{BZ}).\\
    

Finally, we can prove the existence of weak solution in $(x,t)$ coordinates.

\begin{prop}\label{propth} Given a function  $J(x,t)\in (L^2\cap L^\infty\cap C^\alpha)([0,T],\R)$ and initial data $\theta_0(x)$ and $\theta_1(x)$ as in \eqref{data}. Then
\eqref{waveAx} has a weak solution $\theta(x,t)=\theta^J(x,t)$ in the following sense
\begin{align}
    \int_0^T\int_\R\theta_t\phi_t-\big[\gamma_1-\frac{h^2(\theta)}{g(\theta)}\big]\theta_t\phi\,dx\,dt=\int_0^T\int_\R (c(\theta)\theta_x)(c(\theta)\phi)_x+h(\theta)J\phi\,dx\,dt,
\end{align}
for any function $\phi\in H^1_0(\R\times (0,T))$. Moreover, we have
\[\theta(x,t)\in C^{1/2}([0,T],\R),\quad \theta_x(x,t),\theta_t(x,t)\in L^\infty([0,T],L^2(\R)).\]
The initial data is satisfied in the sense that $\displaystyle\theta(x,0)=\theta_0(x)$ point-wise and $\displaystyle \theta_t(x,0)=\theta_1(x)$ in $L^p_{loc}$ for $p=[1,2).$
\end{prop}

\begin{proof}

We show that the solution constructed as a fixed point of the map in Appendix \ref{app-semilinear} satisfies the weak formulation,
\begin{align*}
    \int_0^T\int_\R\theta_t\phi_t+b(\theta)\theta_t\phi\,dx\,dt=\int_0^T\int_\R (c(\theta)\theta_x)(c(\theta)\phi)_x+hJ\phi\,dx\,dt,
\end{align*}
where $b(\theta)=\frac{h^2(\theta)}{g(\theta)}-\gamma_1.$
The calculations are similar to the one in \cite{BZ}. Rewrite the above equation in terms of the variables $R$ and $S$ defined in \eqref{RS} to get
\begin{align*}
  \int_0^T\int_\R\big[\phi_t-c\phi_x\big]R+\big[\phi_t+c\phi_x\big]S+c'\theta_x(S-R)\phi+b(\theta)(R+S)\phi-2hJ\phi\,dx\,dt=0.
\end{align*}
Using \eqref{forward} and $\displaystyle dxdt=\frac{pq}{2c(1+R^2)(1+S^2)}dXdY,$
we obtain
\begin{align*}
 &   \int_0^T\int_\R\bigg\{\big(-2cY_x\phi_Y\big)R+\big(2cX_x\phi_X\big)S+c'\big[\theta_XX_x+\theta_YY_x\big](S-R)\phi\\
    &\qquad -\big[\gamma_1-\frac{h^2(\theta)}{g(\theta)}\big](R+S)\phi-2hJ\phi\bigg\}\frac{pq}{2c(1+R^2)(1+S^2)}dXdY=0.
\end{align*}
Apply (\ref{defpq}) and (\ref{semisys11}) to get
\begin{align*}
  &  \int_0^T\int_\R\frac{R}{1+R^2}p\phi_Y+\frac{S}{1+S^2}q\phi_X+\frac{c'pq}{8c^2}\bigg(\frac{\sin\omega}{1+S^2}-\frac{\sin z}{1+R^2}\bigg)(S-R)\phi\\
    &\qquad -\frac{pq}{2c}\big[\gamma_1-\frac{h^2(\theta)}{g(\theta)}\big]\frac{R+S}{(1+R^2)(1+S^2)}\phi-\frac{pqhJ}{c(1+R^2)(1+S^2)}\phi\,dX\,dY=0.
\end{align*}
Noticing $\displaystyle \frac{R}{1+R^2}=\frac{\sin \omega}{2}\;\mbox{and}\; \frac{S}{1+S^2}=\frac{\sin z}{2}$ from (\ref{defwz}), we get
\begin{align*}
  &  \int_0^T\int_\R\frac{\sin\omega}{2}p\phi_Y+\frac{\sin z}{2}q\phi_X\\
 & \qquad  +\frac{c'pq}{8c^2}\big(\sin\omega\sin z-\sin\omega \cos^2\frac{z}{2}\tan \frac{\omega}{2}-\sin z \cos^2\frac{\omega}{2}\tan \frac{z}{2}\big)\phi\\
    &\qquad -\frac{pq}{2c}\big[\gamma_1-\frac{h^2(\theta)}{g(\theta)}\big]\big(\cos\frac{\omega}{2}\cos^2\frac{z}{2}\sin \frac{\omega}{2}+\cos^2\frac{\omega}{2}\cos\frac{z}{2}\sin \frac{z}{2}\big)\phi\\
    &\qquad -\frac{pq}{c}h\cos^2\frac{\omega}{2}\cos^2\frac{z}{2}J\phi\,dX\,dY=0.
\end{align*}

Integrating the first two terms by parts and using the equations for $p,q,\omega$ and $z$, we prove that the weak formulation \eqref{thetaweak1} is satisfied.\\

It remains to show the H\"older continuity of $\theta(x,t)$ with exponent $1/2.$ This follows from the fact that
\begin{align*}
    \int_0^t[\theta_t\pm c(\theta)\theta_x]^2\,dt\leq C
\end{align*}
for any $t\in [0,T],$ where the constant $C$ depends only on $t.$ Using the change of coordinates \eqref{forward} (See Appendix \ref{app-semilinear}), we obtain
\begin{align*}
     \int_0^t[\theta_t+c(\theta)\theta_x]^2\,dt&=
    \int_{X_0}^{X_t}(2cX_x\theta_X)^2\frac{1}{2X_t}\,dX\\
    &=\int_{X_0}^{X_t}\big(\frac{4c}{p(1+\cos \omega}\frac{\sin \omega}{4c}p\big)^2\frac{1+\cos \omega}{4c}p\,dX\\
    &=\int_{X_0}^{X_t}\frac{\sin^2\omega}{1+\cos \omega}p\,dX\leq C.
\end{align*}
Similar calculations for $\theta_t-c(\theta)\theta_x$ gives a similar bound. The two bounds together imply the square integrability of $\theta_t$ and $\theta_x$ hence Sobolev embedding implies the H\"older continuity of $\theta(x,t)$ with exponent $1/2.$\\

Finally we show a bound for the energy $E(t)$ defined as
\begin{align}
  E(t):=\frac{1}{2}\int_\R\theta_t^2+c(\theta)^2\theta_x^2\,dx.
\end{align}
For any fixed $0<T<\infty,$ let $\Omega_T:=\R\times [0,T].$ For any function $J(x,t)\in L^\infty\cap L^2\cap C^\alpha(\Omega_T),$ the energy of a weak solution $\theta$ of \eqref{waveAx} satisfies a prior bound. More precisely,
the energy satisfies the following bound
\begin{align}\label{energybound}
    E(t)\leq C_E,
\end{align}
for some $C_E$ depending on $E(0)$ and $J.$ The proof is similar to the one in \cite{CHL20} and is provided in Appendix \ref{Appenergy}.

The estimate obtained, 
\beq\label{ET}
\frac{1}{2}\max_{0\leq t
\leq T}E(t)\leq E(0)+C\int_0^T\int_{-\infty}^\infty|J|^2\,dxdt
\eeq
for some constant $C$. This implies that $\theta_t(\cdot, t)$ and $\theta_x(\cdot, t)$ are both square integrable functions in $x$, so are $R$ and $S$.

The proof that the solution satisfies the initial condition follows by the same argument in \cite[Theoerem 1]{BZ}. We omit the proof here.
\end{proof}

\section{A brief review of parabolic differential operators with non-constant H\"older coefficients.}\label{HK}
In this section we summarise relevant results from first chapter in \cite{Fri} in terms of the specific form of equations appeared in this paper for direct future usage.

 Let the differential operator $\mathcal{L}$ be defined as
\[\mathcal{L}:=\partial_t-g(\theta)\partial_{xx}+\gamma_1,\]
where $g(\theta)$ is a strictly positive smooth function and $\theta$ is H\"older continuous with exponent $1/2$ with respect to $x$ and $t,$ and $\gamma_1$ is a positive constant. Consider the  differential equation
\begin{align}\label{diff}
    \mathcal{L}\,\omega&=0.
\end{align}


\par

Note that, for any fixed $(\xi,\tau)$,  the heat kernel of the operator
 \[\mathcal{L}_0^{\xi,\tau}:=\partial_t-g(\theta(\xi,\tau))\partial_{xx}\]
 is
\begin{align}\label{heat}
    H^{\xi,\tau}(x-\xi,t-\tau)=\frac{1}{2\sqrt{g(\theta(\xi,\tau))}\sqrt{t-\tau}}e^{-\frac{(x-\xi)^2}{4g(\theta(\xi,\tau))(t-\tau)}}.
\end{align}

\begin{remark}
    The superscripts in $\mathcal{L}_0^{\xi,\tau}$ and $H^{\xi,\tau}(x-\xi,t-\tau)$ indicate the dependence  on $(\xi,\tau)$ via $g(\theta(\xi,\tau)).$
\end{remark}

 Several results established in Chapter 1 of  \cite{Fri} will be recalled and used. These include Theorems 8-11, displays (4.15), (6.12) and (6.13) in \cite{Fri}.
\begin{prop}
There exists a function $\Phi$ such that $\Gamma$ given by
\begin{align}\label{GGGamma}
\Gamma(x,t,\xi,\tau)&=H^{\xi,\tau}(x-\xi,t-\tau)
+\int_\tau^t\int_\mathbb{R}H^{y,s}(x-y,t-s)\Phi(y,s;\xi,\tau)\,dy\,ds
\end{align}
 satisfies \eqref{diff}. Moreover, one has
    \begin{align}\label{phiest}
|\Phi(y,s;\xi,\tau)|\leq \frac{const}{(s-\tau)^{5/4}}e^{\frac{-d(y-\xi)^2}{4(s-\tau)}},
\end{align}
where $d$ is a constant depending on $\|g\|_{L^\infty(\R)}$ and $\gamma_1.$
\end{prop}


Set $\displaystyle \Omega_T:=\R\times (0,T]$ for some $T>0$ and consider the Cauchy problem
\begin{align}\label{chy}
    &\mathcal{L}\,\omega(x,t)=f(x,t),\quad \text{on}\quad\Omega_T\\\label{chyinit}
    &\omega(x,0)=\phi(x),\quad \text{at} \quad t=0
\end{align}
where $f$ is H\"older continuous on $ \overline\Omega_T$ and $\phi$ is continuous on $\R$. It is shown in \cite[Theorem 12]{Fri} that the function
 \begin{align}
     \omega(x,t)=\int_R\Gamma(x,t,\xi,0)\phi(\xi)\,d\xi+\int_0^t\int_R\Gamma(x,t;\xi.\tau)f(\xi,\tau)\,d\xi\,d\tau
 \end{align}
is a {\em classical} solution of the Cauchy problem \eqref{chy} and \eqref{chyinit}. Moreover,
\begin{align}\label{Gammaest}
    |\Gamma(x,t;\xi,\tau)|\lesssim \frac{1}{\sqrt{t-\tau}}e^{-\frac{d(x-\xi)^2}{4(t-\tau)}}\approx H(x-\xi,t-\tau),
\end{align}
\begin{align}\label{Gamma_xest}
    |\Gamma_x(x,t;\xi,\tau)|\lesssim\frac{1}{{t-\tau}}e^{-\frac{d(x-\xi)^2}{4(t-\tau)}}\approx\frac{1}{\sqrt{t-\tau}}H(x-\xi,t-\tau),
\end{align}
where $d$ is a constant depending on $g$ and $\lesssim$ and $\approx$ mean $\leq$ up to a constant and $=$ up to a constant, respectively. In both cases the constant is uniform in $(x,t,\xi,\tau).$  \\


We now apply the above results from \cite{Fri} to get a preliminary result for later usage.
Set
\begin{align*}
M_f(x,t)&:=\int_0^t\int_\R\Gamma(x,t;\xi,\tau)f(\xi,\tau)\,d\xi\,d\tau,\\
    M_{f,x}(x,t)&:=\int_0^t\int_\R\Gamma_x(x,t;\xi,\tau)f(\xi,\tau)\,d\xi\,d\tau.  
\end{align*}
\begin{prop}\label{GammaL2} 
If $f(x,t)\in {L^\infty((0,T),L^2(\R))},$
then
\begin{align}\label{LinfMf}
\|M_f\|_{L^\infty(\Omega_T)}\lesssim T^{1/4}\|f\|_{L^2(\Omega_T)},\;&\; \|M_{f,x}\|_{L^\infty(\Omega_T)}\lesssim  T^{1/4}\|f\|_{L^\infty((0,T),L^2(\R))},\\\label{L2Mf}
\|M_f\|_{L^2(\Omega_T)}\lesssim T\|f\|_{L^2(\Omega_T)},\;&\mbox{ and }\; \|M_{f,x}\|_{L^2(\Omega_T)}\lesssim \sqrt{T}\|f\|_{L^2(\Omega_T)}.
\end{align}
\end{prop}
\begin{proof}
To estimate the $L^\infty$ norm we use \eqref{Gammaest}, \eqref{Gamma_xest} and Cauchy–Schwarz inequality,
\begin{align*}
    |M_f|&\lesssim \bigg(\int_0^t\int_\R \frac{1}{{t-\tau}}e^{-\frac{d(x-\xi)^2}{2(t-\tau)}}\,dx\,d\tau\bigg)^{1/2}\|f\|_{L^2(\Omega_T)}\lesssim T^{1/4}\, \|f\|_{L^2(\Omega_T)}.
\end{align*}
Taking the sup over $\Omega_T,$ we get the first estimate in \eqref{LinfMf}. Similarly,
\begin{align*}
    |M_{f,x}|&\lesssim \bigg(\int_0^t\int_\R \frac{1}{{|t-\tau|}^{2-2r}}e^{-\frac{d(x-\xi)^2}{2(t-\tau)}}\,dx\,d\tau\bigg)^{1/2}\bigg(\int_0^t\int_\R \frac{1}{{|t-\tau|}^{2r}}f^2\,dx\,d\tau\bigg)^{1/2}\\
    &\lesssim \bigg(\int_0^t \frac{1}{{t-\tau}^{\frac{3}{2}-2r}}\,d\tau\bigg)^{1/2} \bigg(\int_0^t \frac{1}{{|t-\tau|}^{2r}}\,dx\,d\tau\bigg)^{1/2}\|f\|_{L^\infty((0,T),L^2(\R))}.
\end{align*}
For $r=\frac{3}{8},$ we get the second estimate in \eqref{LinfMf}.\\

To estimate the $L^2$ norm we use \eqref{Gammaest}, \eqref{Gamma_xest} and the Young's convolution inequality with $r=2$, $p=1$, and $q=2$. On $\Omega_T,$
\begin{align*}
\|M\|_{L^2}\lesssim & \bigg\|\int_0^T\int_\R\frac{1}{\sqrt{t-\tau}}e^{-\frac{d(x-\xi)^2}{4(t-\tau)}}|f(\xi,\tau)|\,d\xi\,d\tau\bigg\|_{L^2}\\
&\lesssim\|H\ast f\|_{L^2}\leq \|H\|_{L^{1}}\|f\|_{L^2}=C\, T\,\|f\|_{L^2},\\
\|M_x\|_{L^2}\lesssim & \bigg\|\int_0^T\int_\R\frac{1}{{t-\tau}}e^{-\frac{d(x-\xi)^2}{4(t-\tau)}}|f(\xi,\tau)|\,d\xi\,d\tau\bigg\|_{L^2}\\
&\lesssim\|\frac{1}{\sqrt{t}}H\ast f\|_{L^2}\leq \|\frac{1}{\sqrt{t-\tau}} H\|_{L^{1}}\|f\|_{L^2}=C\, \sqrt{T}\,\|f\|_{L^2}.
\end{align*}
This completes the proof.
\end{proof}
\begin{prop}\label{GammaCalpha} 
If $f(x,t)\in {L^\infty((0,T),L^2(\R))},$
then
\begin{align}\label{CalphaMf}
M_f\,,\, M_{f,x}\in C^\alpha(\Omega_T),
\end{align}
for $\alpha\in (0,1/4).$
\end{prop}
\begin{proof}
We prove the Holder estimate for $M_f$. The estimate for $M_{f,x}$ is slightly more involved, the same estimate is proven in \cite[Section 2.2.2]{MS}. The restriction on the exponent $\alpha<1/4$ appears only for $M_{f,x}$.
   Recall,
   \begin{align*}
M_f(x,t)&:=\int_0^t\int_\R\Gamma(x,t;\xi,\tau)f(\xi,\tau)\,d\xi\,d\tau\\
       &=\int_0^t\int_{\R}H^{\xi,\tau}(x-\xi,t-\tau)f(\xi,\tau)\,d\xi\,d\tau\\
    &+\int_0^t\int_{\R}\bigg[\int_\tau^t\int_{\R}H^{y,s}(x-y,t-s)\Phi(y,s;\xi,\tau)\,dy\,ds\bigg]f(\xi,\tau)\,d\xi\,d\tau\\
    &=:I(x,t)+II(x,t).
   \end{align*}
   For the term $I(x,t):$
   \begin{align*}
       \frac{|I(x_2,t)-I(x_1,t)|}{|x_2-x_1|^\alpha}\lesssim \int_0^t\int_\R\frac{|z-\xi|}{|t-\tau|^{3/2}}e^{-\frac{d(z-\xi)^2}{4(t-\tau)}}|x_2-x_1|^{1-\alpha}f(\xi,\tau)\,d\xi d\tau\\
       \lesssim \bigg[\int_0^t\int_\R\frac{|z-\xi|^2}{|t-\tau|^{3-\frac{3}{4}}}e^{-\frac{d(z-\xi)^2}{2(t-\tau)}}\,d\xi d\tau\bigg]^{1/2}\bigg[\int_0^t\int_\R\frac{1}{(t-\tau)^{3/4}}f^2(\xi,\tau)\,d\xi\,d\tau\bigg]^{1/2},
   \end{align*}
   for some $z\in(x_1,x_2).$ By the change of variable $u=\frac{z-\xi}{\sqrt{t-\tau}},$ we have
   \begin{align*}
\frac{|I(x_2,t)-I(x_1,t)|}{|x_2-x_1|^\alpha}&\lesssim \int_0^t\int_\R\frac{u^2}{|t-\tau|^{\frac{3}{4}}}e^{-\frac{d\,u^2}{2}}\,d\xi\,d\tau\bigg]^{1/2}\bigg[\int_0^t\int_\R\frac{1}{(t-\tau)^{3/4}}f^2(\xi,\tau)\,d\xi\,d\tau\bigg]^{1/2}\\
&\lesssim T^{1/4}\|f\|_{L^\infty((0,T),L^2(\R))}.
\end{align*}
For the term $II(x,t):$ We use the change of variables $y=u+\xi$ and $s=v+\tau,$ and write
\begin{align*}
    II=\int_0^t\int_\R\int_0^{t-\tau}\int_\R H(x-u-\xi,t-v-\tau)\Phi(u+\xi,v+\tau,\xi,\tau)f(\xi,\tau)\,du\,dv\,d\xi\,d\tau.
\end{align*}
Using the estimate \eqref{phiest}, we have
\begin{align*}
    &\frac{|II(x_2,t)-II(x_1,t)|}{|x_2-x_1|^\alpha}\\
    &\lesssim \int_0^t\int_\R\bigg[\int_0^{t-v}\int_\R\frac{H(x_2-u-\xi)-H(x_1-u-\xi)}{|x_2-x_1|^\alpha}f(\xi,\tau)\,d\xi\,d\tau\bigg]\frac{1}{v^{5/4}}e^{\frac{-d(u)^2}{4v}}\,du\,dv\\
    &\lesssim T^{1/4}\|f\|_{L^\infty((0,T),L^2(\R))}
    \int_0^t\int_\R\frac{1}{v^{5/4}}e^{\frac{-d(u)^2}{4v}}\,du\,dv\\
    &\lesssim T^{1/2}\|f\|_{L^\infty((0,T),L^2(\R))}.
\end{align*}
Due to the space-time scaling of the heat kernel, one can show the Holder estimate in time for $\alpha<1/2.$
\end{proof}
\section{Existence of a solution $v^J$ for $\eqref{veq}$.}\label{J2v}
Recall that $\theta=\theta^J$ is the solution of wave equation (\ref{waveJ}) depending on $J$. 
We now consider Cauchy problem of   \eqref{vJ}
\begin{align}\label{vveq}
& v_t=g(\theta^J)v_{xx}+h(\theta^J)\theta^J_t,\\
 \label{stvinit}
  &  v(x,0)=v_0(x)
\end{align}
and denote the solution by $v^J$.


\begin{prop}\label{propv}
Let  $v_0(x)$ be defined as $v'_0(x)=u_0(x).$ For any $T\in (0,\infty),$ there exists a function $$v^J(x,t)\in L^2((0,T],H^1(\R))$$ that satisfies \eqref{vveq} and \eqref{stvinit} in the sense that
\begin{align}\label{vweak}
     \int_0^T\int_R v^J\phi_t-v^J_x(g(\theta^J)\phi)_x+h(\theta^J)\theta^J_t\phi\,dx\,dt=0
\end{align}
for any $\phi\in H^1_0((0,T],\R)$ and, as $t\to 0^+$,
\begin{align}\label{limv}
v^J(x,t)&\to v_0(x) \mbox{ point-wise, }  ,\\\label{limv_x}
v_x^J(x,t)&\to v'_{0}(x)=u_0(x) \mbox{ almost everywhere. }
\end{align}
 Moreover, adding the initial data,
\[v^J, v^J_x\in C^\alpha(\R,(0,T])\cap L^\infty(\R, [0,T]) \mbox{ for any } 0\leq \alpha<1/4.\]
\end{prop}

\begin{proof} { For simplicity, we will drop the subscript $J$ in the proof.\\ 

Since $\theta_t$ is generally not H\"older continuous, to apply the results in Section \ref{HK}, we let $\theta_t^\epsilon$ be the mollification of $\theta_t$ for $\epsilon>0$ small.} It is  known that $\theta_t^\epsilon\in C^\infty_c(\Omega_T)$ and, as $\epsilon\to 0$, $\theta_t^\epsilon\to \theta_t$  in $L^2(\Omega_T).$ Denote the solution of
\begin{align}\label{stv}
&v_t=g(\theta)v_{xx}+h(\theta)\theta_t^\epsilon
\end{align}
with the same initial condition (\ref{stvinit})
by $v^\epsilon(x,t).$ As discussed in Section \ref{HK}, $v^\epsilon(x,t)$ is 
a classical solution and can be written explicitly as
\begin{align}\label{vepsilon}
    v^\epsilon(x,t)=\int_\R \Gamma^0(x,t,\xi,0)v_0(\xi)\,d\xi+\int_0^t\int_\R\Gamma^0(x,t;\xi,\tau)h(\theta)\theta_\tau^\epsilon(\xi,\tau)\,d\xi\,d\tau,
\end{align}
{where $\Gamma^0$ is the  kernel of the operator ${\mathcal L}_0:=\partial_t-g(\theta(x,t))\partial_{xx}$.
Note that the operator ${\mathcal L}_0$ is the same as  ${\mathcal L}$ in (\ref{diff}) with $\gamma_1=0$. We comment that all estimates in Section \ref{HK} for $\Gamma$ still hold true for $\Gamma^0$ with possibly different constants.}\\

Clearly, $v^\epsilon$ in \eqref{vepsilon} satisfies the weak formulation \eqref{vweak} since it is a classical solution to \eqref{stv} and \eqref{stvinit}. We have
\begin{align}\label{epsilonweak}
   \int_0^T\int_R v^\epsilon\phi_t-v^\epsilon_x(g\phi)_x+h\theta_t^\epsilon\phi\,dx\,dt=0 
\end{align}
for all $\phi\in H^1_0(\R\times (0,T)).$
At this point, we claim that the expressions
\begin{align}\label{compv}
    v(x,t)=\int_\R \Gamma^0(x,t,\xi,0)v_0(\xi)\,d\xi+\int_0^t\int_\R\Gamma^0(x,t;\xi,\tau)h(\theta)\theta_\tau(\xi,\tau)\,d\xi\,d\tau
\end{align}
and
\begin{align}\label{compv_x}
    v_x(x,t)=\int_\R \Gamma^0_x(x,t,\xi,0)v_0(\xi)\,d\xi+\int_0^t\int_\R\Gamma^0_x(x,t;\xi,\tau)h(\theta)\theta_\tau(\xi,\tau)\,d\xi\,d\tau
\end{align}
are the limits of $v^\epsilon(x,t)$ and $v_x^\epsilon(x,t),$ respectively, in $L^2(\Omega_T)$ sense, and hence, $v(x,t)$
 is a weak solution of \eqref{vveq}.\\

 Subtract (\ref{compv}) from (\ref{vepsilon}) and apply Proposition \ref{GammaL2} and estimate \eqref{Gammaest} to get
 \begin{align*}
    |v^\epsilon-v|\leq \int_0^T\int_\R|\Gamma^0||h||\theta^\epsilon_\tau-\theta_\tau|\,d\xi\,d\tau
\end{align*}
and
\begin{align*}
    \|v^\epsilon-v\|_{L^2(\Omega_T)}&\lesssim \|H\ast (|h||\theta^\epsilon_\tau-\theta_\tau|)\|_{L^2(\Omega_T)}\\
    &\lesssim \|H\|_{L^1(\Omega_T)}\|\theta^\epsilon_\tau-\theta_\tau\|_{L^2(\Omega_T)}\to 0\;\mbox{ as }\; \epsilon\to 0.
\end{align*}
 Similarly, by Proposition \ref{GammaL2} and estimate \eqref{Gamma_xest} we have
 \begin{align*}\begin{split}
    \|v_x^\epsilon-v_x\|_{L^2(\Omega_T)}&\lesssim \|\frac{1}{\sqrt{t}}H\ast (|h||\theta^\epsilon_\tau-\theta_\tau|)\|_{L^2(\Omega_T)}\\
    &\lesssim
    \|\frac{1}{\sqrt{t-\tau}}H\|_{L^1(\Omega_T)}\|\theta^\epsilon_\tau-\theta_\tau\|_{L^2(\Omega_T)}\to 0\;\mbox{ as }\; \epsilon\to 0.
    \end{split}
\end{align*}
Taking $\epsilon\to 0$ in \eqref{epsilonweak} we obtain \eqref{vweak}. Hence, the weak formulation \eqref{vweak} is satisfied as a limit of the classical solution $v^\ve$ to the initial value problem \eqref{stv},\eqref{stvinit}.\\

For the initial data, the first limit \eqref{limv} follows from \cite{Fri}.
The second limit \eqref{limv_x} can be shown by considering the equation satisfied by the first term of \eqref{compv_x} that is by letting $u^0(x,t)=v^0_x(x,t):=\int_\R \Gamma^0_x(x,t,\xi,0)v_0(\xi)\,d\xi.$ Then
\[u^0_t-(g(\theta)u^0_x)_x=0\]
\[u^0(x,0)=u_0(x)\in H^1(\R).\]

It is known that the solution $u^0\in C([0,T],L^2(\R))$ and $u^0(x,0)=u_0(x)$ almost everywhere \cite{Evans}. Hence, we obtain
\[v_x(x,0)=v'_0(x), \mbox{ almost everywhere }\]

Finally, to show $v, v_x\in L^\infty\cap C^\alpha(\R\times (0,T]),$ we use Propositions \ref{GammaL2} and \ref{GammaCalpha}. As mentioned in the proof of Proposition \ref{GammaCalpha}, we refer the reader to \cite[Section 2.2.2]{MS} for the full details.
\end{proof}
\section{Existence of a solution $J.$}\label{FixedPoint}
\subsection{Fixed Point Argument:}
{ Recall in Sections \ref{J2theta} and \ref{J2v}, for any $J(x,t)\in C^\alpha\cap L^2\cap L^\infty$ for some $\alpha>0$, we solve $\theta^J$ from system (\ref{thetaeq}) with $A_x$ replaced by $J$,  then solve $v^J$ (and hence $u^J$) from system (\ref{ueq}) with $\theta$ replaced by $\theta^J$, and  we show that
\beq\label{6.1}
v^J, u^J \in (L^2\cap L^\infty)(\R\times [0,T])\cap C^\alpha(\R\times (0,T])\;\mbox{ and }\;
\theta^J_t, \theta^J_x\in L^\infty([0,T],L^2(\R)).\eeq

We now solve $A^J$ from system (\ref{AAeq}) with $(\theta,v,A_x+A_{0,x})$ replaced by $(\theta^J,v^J,J)$. With all these preparations, we will then define a mapping $\mathcal M(J)$ so that its fixed point gives rise to a solution of our original Cauchy problem.}

%
%

In view of system (\ref{AAeq}), define the operator
\[\mathcal{L}^J:=\partial_t-g(\theta^J)\partial_{xx}+\gamma_1.\]
The function $A$, introduced in (\ref{VarA}), satisfies \eqref{AAeq} recast below
\begin{align}\label{Aeq}
    \mathcal{L}^J A=&F(\theta^J,\theta^J_t,\theta^J_x, v^J)+g'(\theta^J)\theta^J_x J,
\end{align}
where $F(\theta,\theta_t,\theta_x, v)=f(\theta,\theta_x, v_x)+G(\theta,\theta_t,\theta_x, v_x)$ with
$f$ and $G$ given in from \eqref{Fun}, along with the initial data\begin{align}\label{A_0}
    A(x,0)=0.
\end{align} 



Formally, following the discussion in Section \ref{HK}, $A$ can be expressed as
\begin{align}\label{Aexpression}
    A(x,t)=\int_0^t\int_\R\Gamma(x,t;\xi,\tau)\big[F(\theta^J,v^J)+g'(\theta^J)\theta^J_{\xi}J\big](\xi,\tau)\,d\xi\,d\tau.
\end{align}
We now define a mapping $\mathcal{M}$ by
\begin{align}\begin{split}\label{5.6}
    &\mathcal{M}( J)(x,t):=A_x(x,t)+J_{0}\\
&=
\int_0^t\int_\R\Gamma_x(x,t;\xi,\tau)\big[F(\theta^J,v^J)+g'(\theta^J)\theta^J_\xi J\big]\,d\xi\,d\tau+J_0.
\end{split}
\end{align}

The goal is to show the existence of a fixed point $J^*=\mathcal{M}(J^*)$ in a suitable space. This will lead to a weak solution $(\theta^{J^*},u^{J^*})$ for \eqref{sysf} and \eqref{sysw}.




We first give a uniform a priori energy estimate for $\mathcal{E}(t)$ defoined in \eqref{calEdef}.
\begin{thm}\label{lemma3.2} For any fixed $T>0$ and for any weak solution $(u(x,t),\theta(x,t))$ of system \eqref{sysf} and \eqref{sysw}, we have, for $t\in[0,T]$, 
\beq\label{engineq2}
\mathcal{E}(t)\leq \mathcal{E}(0) -\iint_{\R\times [0,t]}(\frac{v_t^2}{g^2}+\theta^2_t)\,dxdt.
\eeq
\end{thm}
The proof is the same as the one in \cite{CHL20}. One can find it in the appendix \eqref{appE}.
\begin{cor}
For any weak solution of system \eqref{sysf} and \eqref{sysw}, there exists a constant $C_0$ depending on $\mathcal{E}(0),\|J'_0\|_{L^2(\R)},$ and $\|J_0\|_{L^1(\R)}$, such that
\beq\label{Gfbound}
\|G\|_{L^\infty(\Omega_T)}\leq C_0,\qquad \|f\|_{L^\infty([0,T],L^2(R))}\leq C_0,\eeq
\end{cor}
The proof is straightforward using \eqref{calEdef}, \eqref{tecinitial}, \eqref{positiveDamping},
and definition of $G$ and $f$ in \eqref{Fun}.

Now we fix an arbitrary $T>1$ once and for all and
consider the following spaces over $\overline{\Omega}_T:=\R\times [0,T].$ 
We use the following Banach space $L^*$, which includes all functions $f$ in $\bigcap\limits_{p\in[2+a,\infty)} L^p$, for some fixed $a>0$, with a finite norm
\[
\|f\|_{L^*(\overline{\Omega}_T)}=\sup\limits_{p\in[2+a,\infty)}\|f\|_{\bar L^p(\overline{\Omega}_T)}< \infty,
\]
and 
\begin{align*}
    N_T:=\bar C^\alpha\cap \bar L^2\cap \bar L^\infty(\overline{\Omega}_T),
\end{align*}
with
\begin{align*}
 \|S\|_{N_T(\overline\Omega_T)}=\max\left\{\|S\|_{\bar L^\infty(\overline\Omega_T)}, \|S\|_{\bar L^2(\overline\Omega_T)}, \|S\|_{\bar C^\alpha(\overline\Omega_T)}\right\},   
\end{align*}
where $\alpha\in (0,1/4)$ and
\[\|S\|_{\bar L^\infty(\overline\Omega_T)}=\|e^{-\lambda t}S(x,t)\|_{L^\infty(\overline\Omega_T)},\]

\[\|S\|_{\bar L^p(\overline\Omega_T)}=\|e^{-\lambda t}S(x,t)\|_{L^p(\overline\Omega_T)},\]
and

\[\|S\|_{\bar C^\alpha(\overline\Omega_T)}=\sup_{\|(h_1,h_2)\|>0}e^{-\lambda t}\frac{|S(x+h_1,t+h_2)-S(x,t)|}{\|(h_1,h_2)\|^\alpha},\]
for some $\lambda=\lambda(\mathcal{E}(0),T)>0$ sufficiently large that will be determined later. Let 
\beq\label{kT}
k_T=2(\|J_0\|_{C^\alpha\cap L^2\cap L^\infty}+\max\{C_0,C_0^2\}T^2).
\eeq 
We define
\begin{align*}
    K_T&:=\big\{\mathcal{S}(x,t)\in N_T:\|\mathcal{S}\|_{N_T}\leq k_T,\, \mathcal{S}(x,0)=J_0(x)\big\}.
\end{align*}
By \eqref{tecinitial} and the Sobolev embedding, it is clear that $\|J_0(x)\|_{N_T(\R)}<\|J_0(x)\|_{C^\alpha\cap L^2\cap L^\infty}.$ 

Furthermore, for any fixed $T$, it is easy to show that the $N_T$ norm and the $C^\alpha\cap L^2\cap L^\infty$ norm are equivalent. In \cite{BZ}, a similar norm was used to prove the existence.
\begin{cor}
    $K_T$ is compact in $L^*$ on any $\overline\Omega_T.$
\end{cor} 
\begin{proof} The proof of this corollary is given in Section 6.3 of \cite{CHL20}, using the Frechet-Kolmogorov theorem and the fact that the $\bar L^p, \bar L^\infty,  \bar C^\alpha$ norms are equivalent to $L^p, L^\infty, C^\alpha$ norms in $t\in[0,T]$, respectively.
\end{proof}

We now recall the Schauder Fixed Point Theorem that will be applied to complete our analysis.

\begin{thm}[Schauder Fixed Point Theorem]\label{sch}
Let $E$ be a Banach space, and let $K$ be a convex set in $E$. Let  $\mathcal T: K\to K$ be a continuous map such that $\mathcal T(K) \subset K$, where
$K$ is a compact subset of $E$. Then $\mathcal T$ has a fixed point in $K$.
\end{thm}

The main step  is to verify the two assumptions of Theorem \ref{sch}, that is
\bigskip
\begin{enumerate}[i.]
    \item \emph{The continuity of the map $\mathcal{M}:K_T\to K_T.$}\smallskip\\
    This can be verified using the same argument as in \cite{CHL20}. The idea is to use the change of coordinates and the semi-linear system introduced previously along with the regularity of the transformation that preserves the continuity of the map. We refer the reader to \cite{CHL20}.
    \bigskip
    \item \emph{The inclusion $\mathcal{M}(K_T)\subset K_T.$
    }
 \end{enumerate}
 Now we prove (ii).\\
 
 The following proposition, with the help of \cite[Propositions 1.1 and 1.2]{MS}, is the key estimate to show that for any $T>0$ we have the inclusion \[\mathcal{M}(J):K_T\to K_T.\]

To state the proposition, let $G\in L^\infty(\Omega_T)$ and $f\in$ $L^\infty((0,T),L^2(\R))$ and define
\begin{align*}
    M_{G}(x,t):=\int_0^t\int_{\R}\Gamma(x,t;\xi,\tau)G(\xi,\tau)\,d\xi\,d\tau,
\end{align*}
\begin{align*}
M_{f}(x,t):=\int_0^t\int_{\R}\Gamma(x,t;\xi,\tau)f(\xi,\tau)\,d\xi\,d\tau.
\end{align*}
and
\begin{align*}
    M_{G,x}(x,t):=\int_0^t\int_{\R}\Gamma_x(x,t;\xi,\tau)G(\xi,\tau)\,d\xi\,d\tau,
\end{align*}
\begin{align*}
M_{f,x}(x,t):=\int_0^t\int_{\R}\Gamma_x(x,t;\xi,\tau)f(\xi,\tau)\,d\xi\,d\tau.
\end{align*}
\begin{prop}\label{mainlemma} Assume $f$ and $G$ satisfy the bounds in \eqref{Gfbound}. For any given $T>1$, we have
$$\displaystyle \max\{\|M_{G,x}(x,t)\|_{N_ T},\|M_{f,x}(x,t)\|_{N_ T}\}\leq \max\{C_0,C_0^2\}T^2.$$

\end{prop}

\begin{proof}
The proof follows directly from
the following estimates that are proven in \cite[Sections 2 and 4]{MS}
\begin{align}\label{infbound}
\begin{split}
\|M_{G,x}\|_{L^\infty(\Omega_ T)}&\lesssim  T^{1/2}\|G\|_{L^\infty(\Omega_T)},\\
\|M_{f,x}\|_{L^\infty(\Omega_ T)}&\lesssim  T^{1/4}\|f\|_{L^\infty((0,T),L^2(R))},
\end{split}
\end{align}
\begin{align}\label{alphbound}
\begin{split}
\|M_{G,x}\|_{C^\alpha(\Omega_ T)}&\lesssim  T^{r}\|G\|_{L^\infty(\Omega_T)},\\
\|M_{f,x}\|_{C^\alpha(\Omega_ T)}&\lesssim  T^{s}\|f\|_{L^\infty((0,T),L^2(R))},
\end{split}
\end{align}
for some fixed $0<r,s<1.$ And
\begin{align}\label{L2bound}
\begin{split}
 \|M_{G,x}\|_{L^2(\Omega_ T)}&\leq \frac{1}{2} T^{3/2} \|G\|_{L^\infty(\Omega_T)}\|G_x\|_{L^\infty((0,T),L^1(\R))},\\
    \|M_{f,x}\|_{L^2(\Omega_ T)}&\leq  T^{1/2}\|f\|_{L^\infty((0,T),L^2(R))} 
\end{split}
\end{align}
Now by \eqref{Gfbound},
\begin{align}\label{infbound2}
\begin{split}
\|M_{G,x}\|_{L^\infty(\Omega_ T)}\leq C_0T^{1/2}, \quad 
\|M_{f,x}\|_{L^\infty(\Omega_ T)}\leq C_0T^{1/4},
\end{split}
\end{align}
\begin{align}\label{alphbound2}
\begin{split}
\|M_{G,x}\|_{C^\alpha(\Omega_ T)}\leq C_0 T^r, \quad 
\|M_{f,x}\|_{C^\alpha(\Omega_ T)}\leq C_0 T^s,
\end{split}
\end{align}
\begin{align}\label{L2bound2}
\begin{split}
 \|M_{G,x}\|_{L^2(\Omega_ T)}\leq C^2_0 T^{3/2}, \quad 
    \|M_{f,x}\|_{L^2(\Omega_ T)}\leq C_0T^{1/2}. 
\end{split}
\end{align}


Using these estimates and the equivalence of $N_T$ norm and $C^\alpha\cap L^2\cap L^\infty$ norm when $T$ is given, we obtain the desired estimate.
\end{proof}

The map $\mathcal{M}$ in \eqref{5.6} contains two terms in the integration. The first term is $F^J:=G+f$ and the second term is $\displaystyle g'(\theta^J)\theta^J_\xi J\in L^\infty((0,T),L^2(\R)).$ Proposition \ref{mainlemma} gives a uniform bound on the first term, where the bound is less than $k_T,$ chosen in \eqref{kT}.

To control the second term, a special treatment needed due to the extra explicit dependence on $J.$ 
%
Denote
\[Q(x,t):=\int_0^t\int_\R \Gamma_x(x,t,\xi,\tau)g'\theta_\xi(\xi,\tau)J(\xi,\tau)\,d\xi d\tau.\]
We write
\begin{align*}
    \big|\int_0^t\int_\R& e^{\lambda\tau} \Gamma_x(x,t,\xi,\tau)g'\theta_\xi(\xi,\tau)e^{-\lambda\tau}J(\xi,\tau)\,d\xi d\tau\big|\\
    &\lesssim \|e^{-\lambda\tau} J\|_{L^\infty(\Omega_T)}\,\mathcal{E}(0) \bigg[\int_0^t\int_\R(t-\tau)^{3/4} \Gamma_x^2(x,t,\xi,\tau)\,d\xi\,d\tau\bigg]^{1/2}\bigg[\int_0^t\frac{e^{2\lambda\tau}}{(t-\tau)^{3/4}}\,d\tau\bigg]^{1/2}\\
    &\leq \|e^{-\lambda\tau} J\|_{L^\infty(\Omega_T)}\,\mathcal{E}(0)\,t^{1/8}\bigg[\int_0^t\frac{e^{2\lambda\tau}}{(t-\tau)^{3/4}}\,d\tau\bigg]^{1/2}\\
    &=\|e^{-\lambda\tau} J\|_{L^\infty(\Omega_T)}\,\mathcal{E}(0)\,t^{1/8}\bigg[\int_0^t\frac{e^{2\lambda(t-\tau)}}{\tau^{3/4}}\,d\tau\bigg]^{1/2},
\end{align*}
where we used \eqref{Gamma_xest}.
Multiplying by $e^{-\lambda t},$ we get
\begin{align*}
    \big|e^{-\lambda t}\int_0^t\int_\R& e^{\lambda\tau} \Gamma_x(x,t,\xi,\tau)g'\theta_\xi(\xi,\tau)e^{-\lambda\tau}J(\xi,\tau)\,d\xi d\tau\big|\\
    &\leq \|e^{-\lambda\tau} J\|_{L^\infty(\Omega_T)}\mathcal{E}(0)\,t^{1/8}\bigg[\int_0^t\frac{e^{-2\lambda\tau}}{\tau^{3/4}}\,d\tau\bigg]^{1/2}\\
    &=\|e^{-\lambda\tau} J\|_{L^\infty(\Omega_T)}\mathcal{E}(0)\,t^{1/8}\bigg[\int_0^{1/\lambda}\frac{e^{-2\lambda\tau}}{\tau^{3/4}}\,d\tau+\int_{1/\lambda}^{t}\frac{e^{-2\lambda\tau}}{\tau^{3/4}}\,d\tau\bigg]^{1/2}\\
    &\leq \|e^{-\lambda\tau} J\|_{L^\infty(\Omega_T)}\mathcal{E}(0)\,t^{1/8}\bigg[\frac{1}{\lambda^{1/4}}+\lambda^{3/4}\frac{1}{\lambda}(e^{-1}-e^{-\lambda T})\bigg]^{1/2}\\
    &\leq \sqrt{2} \|e^{-\lambda\tau} J\|_{L^\infty(\Omega_T)}\mathcal{E}(0)\,T^{1/8}\frac{1}{\lambda^{1/8}},
\end{align*}
which yields, for $J\in K_T,$
\begin{align}
    \|Q\|_{\bar L^\infty(\Omega_T)}\leq C_T \frac{1}{\lambda^{1/8}} \|J\|_{\bar L^\infty(\Omega_T)}\mathcal{E}(0)\leq \sqrt{2}\,T^{1/8} \frac{1}{\lambda^{1/8}} k_T\mathcal{E}(0).
\end{align}
Hence, choosing $\lambda>(2\sqrt{2})^8T\mathcal{E}^8(0),$ we obtain the bound
\begin{align}\label{Linf*}
  \|Q\|_{\bar L^\infty(\Omega_T)}\leq k_T/2.  
\end{align}
For the $\bar L^2$ estimate, using \eqref{Gamma_xest} and \eqref{L2bound}

\begin{align*}
  e^{-\lambda t}|Q(x,t)|&\leq e^{-\lambda t}\int_0^t\int_\R \frac{1}{{t-\tau}}e^{-\frac{d(x-\xi)^2}{4(t-\tau)}}|g'||\theta_\xi(\xi,\tau)||J(\xi,\tau)|\,d\xi d\tau\\
  &\leq \|g'\|_{L^\infty} \|J\|_{\bar L^\infty}\int_0^t\int_\R \frac{1}{{t-\tau}}e^{-\frac{d(x-\xi)^2}{4(t-\tau)}}|\theta_\xi(\xi,\tau)|e^{\lambda(\tau-t)}\,d\xi d\tau\\
&=\|g'\|_{L^\infty}\|J\|_{\bar L^\infty}\int_0^t\int_\R \frac{1}{{\tau}}e^{-\frac{d\,\xi^2}{4\tau}}|\theta_\xi(x-\xi,t-\tau)|e^{-\lambda\tau}\,d\xi d\tau.
\end{align*}
Taking the $L^2$ in $x,$
\begin{align*}
    \|e^{-\lambda t}Q(x,t)\|_{L^2(\R)}&\leq \|g'\|_{L^\infty} \|J\|_{\bar L^\infty(\Omega_T)} \|\theta_\xi\|_{L^\infty((0,T),L^2(\R))}\int_0^t\frac{1}{\sqrt{\tau}}e^{-\lambda \tau}\,d\tau\\
    &\leq \|g'\|_{L^\infty}\|J\|_{\bar L^\infty} \mathcal{E}(0)\frac{3}{\sqrt{\lambda}}.
\end{align*}
This gives
\begin{align}
    \|Q\|_{\bar L^2(\Omega_T)}\leq \|g'\|_{L^\infty}k_T \mathcal{E}(0)\frac{3}{\sqrt{\lambda}}\,\sqrt{T}.
\end{align}
Hence, for $\lambda>36\,\mathcal{E}(0)^2\|g'\|_{L^\infty}^2T$ we have
\begin{align}\label{L2*}
    \|Q\|_{\bar L^2(\Omega_T)}\leq k_T/2.
\end{align}
By a similar argument one can show the existence of $\lambda>0$ depending only on $\mathcal{E}(0),\,T$ and bound of $g$ and $g'$ such that
\begin{align}\label{alpha*}
    \|Q\|_{\bar C^\alpha(\Omega_T)}\leq k_T/2.
\end{align}
\begin{remark}
    The weight $e^{-\lambda t}$ introduced in the norm $N_T$ helps to get the inclusion of the map $\mathcal{M},$ particularly, for the term $Q$ treated above. One can see, for example, the $\bar L^\infty$ estimate; we got \[\|Q\|_{\bar L^\infty}\leq C_T\frac{1}{\lambda^{1/8}} \|J\|_{\bar L^\infty}\mathcal{E}(0).\]
    This allows us to choose $\lambda=\lambda(T,\mathcal{E}(0))$ such that $\|Q\|_{\bar L^\infty}<\frac{1}{2}  \|J\|_{\bar L^\infty}.$ In other words, the size of $\|Q\|_{N_T}$ shrinks faster than the size of $\|J\|_{N_T}$ as $\lambda$ gets large.
\end{remark}

Now, Proposition \ref{mainlemma} with \eqref{Linf*},\eqref{L2*} and \eqref{alpha*} show that for any fixed $T>1,$
there exists $\lambda>0$ large enough (depending only on $T, \mathcal{E}(0), \|g\|_{L^\infty}$ and $\|g'\|_{L^\infty}$)
such that $$\mathcal{M}(K_T)\subset K_T.$$ Hence, by the Schauder fixed point theorem, we have a fixed point $J^*=\mathcal{M}(J^*)\in K_T.$ Since for a fixed finite time $T$ the $N_T$ norm and the $L^2\cap L^\infty\cap C^\alpha(\Omega_T)$ norm are equivalent,  $J^*\in L^2\cap L^\infty\cap C^\alpha.$ Therefore, we have the existence of $A(x,t)$ by \eqref{Aexpression}. 
\subsection{The Weak Formulation:}
The next proposition shows that the weak formulation is well-defined. Before that we will introduce $A^\ve(x,t)$ to be the solution to
\begin{align}\label{Aepsilon}
    A^\ve_t-g(\theta)A^\ve_{xx}+\gamma_1A^\ve=F^{\epsilon}+g'\theta_x^\epsilon J.
\end{align}
Here $F^\epsilon,$ defined in \eqref{Fun}, and $\theta_x^\epsilon$ are smooth mollification where the mollification of $F$ is acting only on $\theta_x$ and $\theta_t$. Precisely, $F^\epsilon=F(\theta,\theta_t^\epsilon,\theta_x^\epsilon,v_x)$. The mollification is not needed for $v_x$ since we have shown $\theta\in C^{1/2}(\Omega_T)$ and $v_x\in C^\alpha(\Omega_T)$. Hence, the classical theory discussed in Section \ref{HK} applies.

\begin{prop}\label{propA}
For any $T\in (0,\infty),$ there exists a function $A(x,t)$ such that $$A_x\in L^2\cap L^\infty\cap C^\alpha([0,T],\R)$$
and
\begin{align}\label{Aweak}
    \int_0^T\int_\R A\phi_t-g(\theta) A_x\phi_x-\gamma_1 A\phi\,dx\,dt=-\int_0^T\int_\R (F+g'(\theta)\theta_x J_0)\phi\,dx\,dt,
\end{align}
for any $\phi\in C^\infty_c(\R\times (0,T)).$ More precisely, the weak formulation \eqref{Aweak} is satisfied as a limit of $A^\ve.$
\end{prop}
\begin{proof}
First, for short, we denote $F^{J^*}$ by $F$ and $\theta^{J^*}$ by $\theta.$
The fixed point $J^*$ gives
\begin{align}\label{Express4A}
A(x,t)=\int_0^t\int_\R\Gamma(x,t;\xi,\tau)\big[F+g'(\theta)\theta_\xi J^*\big]\,d\xi\,d\tau.
\end{align}
Now we show that $A$ satisfies \eqref{Aweak}.
We can write the classical solution $A^\epsilon$ of \eqref{Aepsilon} explicitly as
\begin{align}
    A^\epsilon(x,t)=\int_0^t\int_\R\Gamma(x,t;\xi,\tau)\big[F^{\epsilon}+g'\theta_\xi^{\epsilon}J^*\big]\,d\xi\,d\tau.
\end{align}
The weak formulation of the solution to \eqref{Aepsilon} is defined as the following
\begin{align}\label{eweak}
    \int_0^T\int_\R A^\epsilon \phi_t- A^\epsilon_x(g\phi)_x-\gamma_1 A^\epsilon\phi\,dx\,dt=-\int_0^T\int_\R \big[F^{\epsilon}+g'\theta_x^{\epsilon}J^*\big]\phi\,dx\,dt
\end{align}
since $A^\epsilon$ is a classical solution.
The same argument as in Section \ref{J2v} can be applied here to show that $A^\ve, A_x^\ve\to A,A_x$ in the $L^2_{loc}$ sense. This means
taking $\epsilon\to 0$ in \eqref{eweak} and using $J^*-J_0=A_x$ we obtain \eqref{Aweak}.

An application of Proposition \ref{mainlemma} shows 

$$\displaystyle A_x\in L^\infty\cap L^2\cap C^\alpha([0,T],\R).$$ 
\end{proof}
\begin{cor}\label{hatweak}
The function $\hat A(x,t):=A(x,t)+A_0(x)$ satisfies the following identity
  \begin{align}
    \int_0^T\int_\R \hat A \phi_t-g(\theta)\hat A_x\phi_x-\gamma_1 \hat A\phi\,dx\,dt=&-\int_0^T\int_\R \hat F\phi\,dx\,dt
\end{align} 
where
\begin{align}\begin{split}
\hat F&=\hat f+\hat G,\\
\hat f&=[\gamma_1-\frac{h^2}{g}]v_x+\frac{h(\theta)c^2(\theta)}{g}\theta_x,\\
\hat G&=\int_{-\infty}^x[\frac{h'}{g}-\frac{g'h}{g^2}]\theta_t^2-[\gamma_1-\frac{h^2}{g}]'\theta_zv_z-(\frac{h(\theta)c(\theta)}{g})'c(\theta)\theta_z^2\,dz.
    \end{split}  
\end{align}
\end{cor}
\begin{proof}
Using $A=\hat A(x,t)-A_0$ and \eqref{Aweak}, we have
\begin{align}
    \int_0^T\int_\R \hat A \phi_t-g(\theta)\hat A_x\phi_x-\gamma_1 \hat A\phi\,dx\,dt=&\int_0^T\int_\R g(\theta) J'_0\phi-\gamma_1 A_0\phi-F\phi\,dx\,dt\nonumber\\
    =&-\int_0^T\int_\R \hat F\phi\,dx\,dt\nonumber
\end{align}    
\end{proof}


\section{Weak formulations}\label{bdE}
To summarize, from Propositions \ref{propth}, \ref{propv}, \ref{propA} and Corollary \ref{hatweak} we have proved that $(\theta,v,A)$ defines a weak solution in the sense that, for any test function $\phi\in H_0^1((0,T)\times\R)$,
     \begin{align}\label{thetaweakk}
    \int_0^T\int_\R\theta_t\phi_t-\big(\gamma_1-\frac{h^2}{g}\big)\theta_t\phi\,dx\,dt=\int_0^T\int_\R (c(\theta)\theta_x)(c(\theta)\phi)_x+h\hat A_x\phi\,dx\,dt,
\end{align}
\begin{align}\label{vvweak}
    \int_0^T\int_\R v\phi_t-v_x(g\phi)_x+h\theta_t\phi\,dx\,dt=0,
\end{align}
\begin{align}\label{AAweak}
    \int_0^T\int_\R \hat A \phi_t-g(\theta)\hat A_x\phi_x-\gamma_1 \hat A\phi\,dx\,dt=-\int_0^T\int_\R \hat F\phi\,dx\,dt.
\end{align}

Recall,
\begin{align}\begin{split}
\hat F&=\hat f+\hat G,\\
\hat f&=[\gamma_1-\frac{h^2}{g}]v_x+\frac{h(\theta)c^2(\theta)}{g}\theta_x,\\
\hat G&=\int_{-\infty}^x[\frac{h'}{g}-\frac{g'h}{g^2}]\theta_t^2-[\gamma_1-\frac{h^2}{g}]'\theta_zv_z-(\frac{h(\theta)c(\theta)}{g})'c(\theta)\theta_z^2\,dz.
    \end{split}  
\end{align}

Now we show that $(u,\theta)$ with $u=v_x$ satisfies the requirement in definition \eqref{def1} for a weak solution.
By \eqref{vvweak}, for any test function $\eta\in C_c^\infty(\Omega_T)$, choose $\phi=\eta_x$,  we have 
\begin{align}
    \int_0^T\int_\R u\eta_t-(gu_x+h\theta_t)\eta_x\,dx\,dt=0.
\end{align}

Next we establish the relation between $v_t$ and $J$. Precisely, we show that
\begin{align}
    J-\frac{v_t}{g}=0\;\mbox{ almost everywhere.}
\end{align}
Taking $\displaystyle\phi=\frac{\psi_t}{g}$ in \eqref{vweak}, we have
\begin{align}\label{0}
    \int_0^T\int_\R -v_t\frac{\psi_t}{g}-v_t \psi_{xx}+h\theta_t \frac{\psi_t}{g}\,dx\,dt=0,
\end{align}
which is 
\begin{align}\label{1}
    \int_0^T\int_\R \frac{v_t}{g}(\psi_t+g\psi_{xx})-(\frac{h\theta_t }{g}) \psi_t\,dx\,dt=0.
\end{align}
By \eqref{thetaweakk}, if we choose $\displaystyle \phi=\frac{h}{g}\psi,$ then 
 \begin{align}\label{2}\begin{split}
  &  \int_0^T\int_\R\theta_t\frac{h}{g}\psi_t+(\frac{h}{g})'\theta^2_t\psi-\big(\gamma_1-\frac{h^2}{g}\big)\theta_t\frac{h}{g}\psi\,dx\,dt\\
  &\qquad  =\int_0^T\int_\R (c\theta_x)(c\frac{h}{g}\psi)_x+hJ\frac{h}{g}\psi\,dx\,dt.
  \end{split}
\end{align}
Finally, for \eqref{AAweak}, we choose $\phi=\psi_x$ to get 
\begin{align}\label{3}
    &\iint -J\psi_t-gJ\psi_{xx}+\gamma_1J\psi\,dx\,dt\\
    &\qquad =\iint (\frac{h}{g})'\theta_t^2\psi-[\gamma_1-\frac{h^2}{g}] v_{xx}\psi
    -(\frac{hc}{g})'c\theta_x^2\psi-\frac{hc^2}{g}\theta_x\,dx\,dt.\nonumber
\end{align}
Adding \eqref{1}-\eqref{3}, and using that 
\[
\iint (\frac{h^2}{g}-\gamma_1)\frac{1}{g}\psi(v_t-gv_{xx}-\theta_t h)\,dx\,dt=0,
\]
we obtain
\begin{align}\label{identity}
 \iint (J-\frac{v_t}{g})(\psi_t+g\psi_{xx}+(\frac{h^2}{g}-\gamma_1)\psi)\,dx\,dt=0.   
\end{align}
Denote $\displaystyle a:=J-\frac{v_t}{g}.$ Now we show that the weak solution of 
\[
a_t-(ag)_{xx}-\beta a=0\; \mbox{ where }\; \beta=\frac{h^2}{g}-\gamma_1
\]
with $a(x,0)=0$, has only zero solution almost everywhere. To prove that, let $\displaystyle m=ga$, so the problem becomes
\[
m_t-g m_{xx}-(\beta+\frac{g_t}{g})m=0
\]
with $m(x,0)=0$. We need to prove that $m=0$ almost everywhere is the only solution.
    Applying same arguments from Section \ref{HK}, we have an integral formula of the solution and we write the map of solutions as
\[\hat m(x,t)=\int_0^t\int_\R\Gamma(x,t;\xi,\tau)\big(\big[\beta+\frac{g_t}{g}\big]m\big)(\xi,\tau)\,d\xi\,d\tau.\]
Following the same fixed point arguments in Section \ref{FixedPoint}, we have the following estimate for some short time $t$ and some constant $r>0$
\[\|m\|_{L^\infty}<t^r \|m\|_{L^\infty}\]
since $m(x,0)=0$, so $m=0$, a.e. when $t$ is small enough, then for any $t$. Hence,
\[
J=\frac{v_t}{g} \quad \hbox{almost everywhere.}
\]
Furthermore, by \eqref{vvweak} and choosing $\phi=\frac{\varphi}{g}$, we have 
\begin{align}\label{4}
    \int_0^T\int_\R \frac{v_t}{g}\varphi+u_x\varphi+\frac{h\theta_t }{g}\varphi\,dx\,dt=0.
\end{align}
So 
\[
J=\frac{v_t}{g}=u_x+\frac{h\theta_t}{g} \quad \hbox{almost everywhere.}
\]
Note that $J\in K_T$,  $\theta$  satisfies \eqref{energybound} and $h/g$ is uniformly bounded. So we have $u_x\in L^{\infty}([0,T],L^2_{loc}(\Bbb R))\cap L^2([0,T],L^2(\R)).$ Hence we can prove \eqref{u_est} using \eqref{6.1}.



\appendix

\section{proof of proposition \ref{localwave}}\label{app-semilinear}

Similar to the work in \cite{CHL20}, we use the Schauder fixed point theorem to show the existence of a solution of system \eqref{semisys11}-\eqref{semisys55}. We apply the fixed point argument to the following maps obtained by integrating the equations \eqref{semisys11}-\eqref{semisys55}.

\beq\label{hattheta}
\hat{\theta}(X,Y)=\theta(X,\phi(X))+\int_{\phi(X)}^Y\frac{\sin z}{4c}q (X,\bar Y)d\bar Y,
\eeq
\beq\label{hatz}
\begin{split}
\hat{z}(X,Y)=&z(\phi^{-1}(Y),Y)+
\int_{\phi^{-1}(Y)}^X p\Big\{\frac{c'}{4c^2}(\cos^2\frac{w}{2}-\cos^2\frac{z}{2})\\
&+\frac{\frac{h^2(\theta)}{g(\theta)}-\gamma_1}{4c}(\sin w\cos^2\frac{z}{2}+\sin z\cos^2\frac{w}{2})
-\frac{h}{c}J(x_p,t_p)\cos^2\frac{z}{2}\cos^2\frac{w}{2}\Big\}(\bar X,Y)d\bar X ,
\end{split}
\eeq
\beq\label{hatw}
\begin{split}
\hat{w}(X,Y)=&w(X,\phi(X))+\int_{\phi(X)}^Y
q
\Big\{\frac{c'}{4c^2}(\cos^2\frac{z}{2}-\cos^2\frac{w}{2})\\
&+\frac{\frac{h^2(\theta)}{g(\theta)}-\gamma_1}{4c}(\sin w\cos^2\frac{z}{2}+\sin z\cos^2\frac{w}{2})
-\frac{h}{c}J(x_m,t_m)\cos^2\frac{z}{2}\cos^2\frac{w}{2}\Big\}(X,\bar Y)d\bar  Y ,
\end{split}
\eeq
\beq\label{hatp}
\begin{split}
\hat{p}(X,Y)=&p(X,\phi(X))+\int_{\phi(X)}^Y pq\left\{\frac{c'}{8c^2}(\sin z-\sin w)\right.\\
&\left.
+\frac{\frac{h^2(\theta)}{g(\theta)}-\gamma_1}{4c}[\frac{1}{4}\sin w\sin z+\sin^2 \frac{w}{2}\cos^2\frac{z}{2}]-\frac{h}{2c}J(x_m,t_m)\sin w\cos^2\frac{z}{2}\right\}(X,\bar Y)d\bar Y,
\end{split}
\eeq
\beq\label{hatq}
\begin{split}
\hat{q}(X,Y)=&q(\phi^{-1}(Y),Y)+\int_{\phi^{-1}(Y)}^X pq\left\{\frac{c'}{8c^2}(\sin w-\sin z)\right.\\
&\left.
+\frac{\frac{h^2(\theta)}{g(\theta)}-\gamma_1}{4c}[\frac{1}{4}\sin w\sin z+\sin^2 \frac{z}{2}\cos^2\frac{w}{2}]-\frac{h}{2c}J(x_p,t_p)\sin z\cos^2\frac{w}{2}\right\}(\bar X,Y)d\bar X,
\end{split}
\eeq

where it is important to express
\beq\label{xpdef}
x_p(\bar X,Y)=x(\bar X,\phi(\bar X))+\int_{\phi(\bar X)}^Y-\frac{1+ \cos z}{4}q d\tilde Y,
\eeq
\beq\label{tpdef}
t_p(\bar X,Y)=t(\bar X,\phi(\bar X))+\int_{\phi(\bar X)}^Y\frac{1+ \cos z}{4c}q d\tilde Y,
\eeq
\beq\label{xmdef}
x_m(X,\bar Y)=x(\phi^{-1}(\bar Y),\bar Y)+\int_{\phi^{-1} (\bar Y)}^X-\frac{1+ \cos w}{4}p d\tilde X,
\eeq
and
\beq\label{tmdef}
t_m(X,\bar Y)=t(\phi^{-1}(\bar Y),\bar Y)+\int_{\phi^{-1} (\bar Y)}^X-\frac{1+ \cos w}{4c}p d\tilde X.
\eeq

The initial line $t=0$ in the $(x,t)$-plane is transformed to a
parametric curve 
\begin{align*}
 \Gamma_0 :=   \big\{ (X,Y):\,~~Y=\varphi (X)\big\}\subset {\mathbb
R}^2.  
\end{align*}

Next, Denote $V=(\theta,z,\omega,p,q)$ and $\overline{V}(X)=(\theta,z,\omega,p,q)(X,\phi(X))$ or equivalently $\overline{V}(Y)=(\theta,z,\omega,p,q)(\phi^{-1}(Y),Y)$

Following the work in \cite{CHL20} we choose $X:=C_0(\Omega_T)$ and for fixed constant $b>0$
 define the set
$$B_ T:=\{V:\|V\|_{L^\infty\cap C^\alpha(\Omega_ T)}\leq b,V(\Gamma_0)=\overline V\},$$
where $\Omega_ T:=\big\{(X,Y)\in \R^2:\mbox{dist}((X,Y), \Gamma_0)\leq T \big\}.$

Now, using the Schauder fixed point theorem and following the work in \cite{CHL20}, for small enough $ T>0$ we have a fixed point, that is $(\theta,z,\omega,p,q)$ such that $(\hat\theta,\hat z,\hat\omega,\hat p,\hat q)=(\theta,z,\omega,p,q).$ This gives the local existence of solutions: Proposition \ref{localwave}.\\

\begin{remark}
The technique followed here for the wave, in fact, relies on the finite propagation of the solution. More precisely, we consider a bounded region such that the solution exists out of this region (the far field). We need to show the existence of a solution in this bounded region and then we glue it with the solution in the far field. We refer the reader to \cite{CHL20} for details.
\end{remark}

\section{A bound of the energy $E$}\label{Appenergy}

Apply the Green's Theorem  over the region $D_t$ in Figure \ref{fig_last} to get
\begin{align}\label{eqn3.2}\begin{split}
&\int_{\partial D_t}\frac{1-\cos w}{4}p\,dX-\frac{1-\cos z}{4}q\,dY\\
 &\qquad =-\frac{1}{4}\iint_{D_t}\big[((1-\cos z)q)_X+((1-\cos w)p)_Y\big]\,dXdY.
 \end{split}
\end{align}
Direct computation gives
\begin{align}\label{pplusq}
&((1-\cos z)q)_X+((1-\cos w)p)_Y\\\nonumber
&\quad =-\frac{pq}{c}(\sin \frac{w}{2}\cos \frac{z}{2}+\sin \frac{z}{2}\cos \frac{w}{2})^2-\frac{pq}{c}J (\sin z\cos^2\frac{w}{2}+\sin w\cos^2\frac{z}{2}).
\end{align}

Substituting it into \eqref{eqn3.2}, and using the transformation relation in \eqref{eqnxt}, the following inequality holds,
\beq\label{eqn3.3}
\begin{split}
&\int_a^b (\theta_t^2+c^2(\theta)\theta_x^2)(x,t)\,dx\\
&\quad =\int_{AB\cap (\cos w\neq -1)}\frac{1-\cos w}{4}p\,dX
+\int_{BA\cap (\cos z\neq -1)}\frac{1-\cos z}{4}q\,dY\\
&\quad \leq\int_{AB}\frac{1-\cos w}{4}p\,dX-\frac{1-\cos z}{4}q\,dY\\
&\quad =\int_{DC}\frac{1-\cos w}{4}p\,dX-\frac{1-\cos z}{4}q\,dY-\int_{DA}\frac{1-\cos w}{4}p\,dX\\ 
&\qquad-\int_{CB}\frac{1-\cos z}{4}q\,dY -\frac{1}{4}\iint_{D_t}\frac{pq}{c}(\sin \frac{w}{2}\cos \frac{z}{2}+\sin \frac{z}{2}\cos \frac{w}{2})^2\,dXdY\\
&\qquad -\frac{1}{4}\iint_{D_t}\frac{pq}{c}hJ (\sin z\cos^2\frac{w}{2}+\sin w\cos^2\frac{z}{2})\,dXdY\\
&\quad \leq  \int_{DC}\frac{1-\cos w}{4}p\,dX-\frac{1-\cos z}{4}q\,dY-2\iint_{\mathcal{D}}\theta^2_t\,dxdt-2\iint_{\mathcal{D}}hJ\theta_t\,dxdt\\
&\quad=\int_d^c (\theta_t^2+c^2(\theta)\theta_x^2)(x,0)\,dx-2\iint_{\mathcal{D}}\theta^2_t\,dxdt-2\iint_{\mathcal{D}}hJ\theta_t\,dxdt,
\end{split}
\eeq
where we have used the following fact in the second to the last step: 
\beq\notag
\left|\frac{\partial(X,Y)}{\partial(x,t)}\right|
=\left|
\begin{array}{cc}
\displaystyle X_x& \displaystyle X_t\\
\displaystyle Y_x& \displaystyle Y_t
\end{array}
\right|=-2cX_xY_x=\frac{8}{pq}\frac{1}{1+\cos w}\frac{1}{1+\cos z}.
\eeq

For any $0\leq t
\leq T$,
\[
E(t)\leq E(0)+4  \int_0^t\int_{-\infty}^\infty|J||\theta_t|\,dxdt,
\]
and hence, 
\beq
\frac{1}{2}\max_{0\leq t
\leq T}E(t)\leq E(0)+C_\ve\int_0^T\int_{-\infty}^\infty|J|^2\,dxdt
\eeq
for some constant $C_\ve\to \infty$ as $\epsilon\to 0.$. This implies that $\theta_t(\cdot, t)$ and $\theta_x(\cdot, t)$ are both square integrable functions in $x$, so do $R$ and $S$.

\section{Derivation of equation (\ref{AAeq}) for the quantity $A$}\label{App-A}

Denote
\[
\hat A(x,t)=\int_{-\infty}^xJ(z,t)\,dz, \qquad
A_0=\int_{-\infty}^xJ(z,0)\,dz,\qquad
A(x,t)=\hat A(x,t)-A_0.
\]
By \eqref{J_eq}, one can find 
\begin{align*}
    \hat A_t&={\int_{-\infty}^x\frac{(gJ)_t}{g}\,dz}-{\int_{-\infty}^x\frac{g'\theta_t}{g}J\,dz}\\
    &={(gJ)_x+\int^x_{-\infty}\frac{1}{g}\bigg[[g'\theta_t-\gamma_1 g]J+[h'-\frac{g'h}{g}]\theta_t^2+[\gamma_1g-h^2]u_x+h(\theta)c(\theta)(c(\theta)\theta_x)_x\bigg]\,dz}\\
    &{-\int_{-\infty}^x\frac{g'\theta_t}{g}J\,dz}\\
    &=\big(g(\theta)\hat A_x\big)_x+\int^x_{-\infty}\frac{1}{g}\bigg[[-\gamma_1]gJ+[h'-\frac{g'h}{g}]\theta_t^2+[\gamma_1g-h^2]u_x+h(\theta)c(\theta)(c(\theta)\theta_x)_x\bigg]\,dz
\end{align*}

Integrating by parts, one has 
\begin{align*}\nonumber
    \hat A_t&=g(\theta)\hat A_{xx}-\gamma_1\hat A+g'\theta_xJ+\bigg[\int_{-\infty}^x[\frac{h'}{g}-\frac{g'h}{g^2}]\theta_t^2-[\gamma_1-\frac{h^2}{g}]'\theta_zu-(\frac{h(\theta)c(\theta)}{g})'c(\theta)\theta_z^2\,dz\bigg]\\
    &\quad+[\gamma_1-\frac{h^2}{g}]u+\frac{h(\theta)c^2(\theta)}{g}\theta_x\\
        &:=g(\theta)\hat A_{xx}-\gamma_1\hat A+g'\theta_x J+\hat F( \theta,u).
\end{align*}
After a linear transformation, it is easy to get the equation (\ref{AAeq}) for $A$ with $u=v_x.$



\section{Proof of Theorem \ref{lemma3.2} for energy estimate.\label{appE}}
Now we prove Theorem \ref{lemma3.2} for the energy decay, with the energy defined by
\begin{align}\label{D1}
    \mathcal{E}(t):=\int_\R\theta_t^2+c^2(\theta)\theta_x^2+u^2\,dx.
\end{align}

The proof is the same as the one in \cite{CHL20}. We include a brief proof here and refer interested readers to \cite{CHL20} for details.

\begin{proof} We first consider the bounded region $D_{t}$ in the $(X,Y)$-plane in Figure \ref{fig_last}.
 Denote the bounded region in the $(x,t)$ plane corresponding to $D_t$ by $\mathcal{D}.$

\bigskip
\begin{figure}[ht]
\centering
\begin{tikzpicture}
\draw (-2,0) -- (-2,3) node[anchor=south] {$t$};
\draw (-5,0) -- (2,0) node[anchor=west] {$x$};
\draw (-5,1) -- (2,1) node[anchor=west] {$t$};
\filldraw [gray] (-2.5,1) node[anchor=north] {$\mathcal{D}$};
\filldraw [gray] (0,0) circle (2pt) node[anchor=north] {$(c,0)$};
\filldraw [gray] (-4,0) circle (2pt) node[anchor=north] {$(d,0)$};
\filldraw [gray] (-0.3,1) circle (2pt) node[anchor=south west] {$(b,t)$};
\filldraw [gray] (-3.6,1) circle (2pt) node[anchor=south east] {$(a,t)$};
\draw (0,0) .. controls (-0.4,3/2) .. (-2,2.5) node[anchor=west] {$x^-$};
\draw (-4,0) .. controls (-7/2,3/2) .. (-2,2.5) node[anchor=east] {$x^+$};
\end{tikzpicture}\qquad
\includegraphics[width=.35\textwidth]{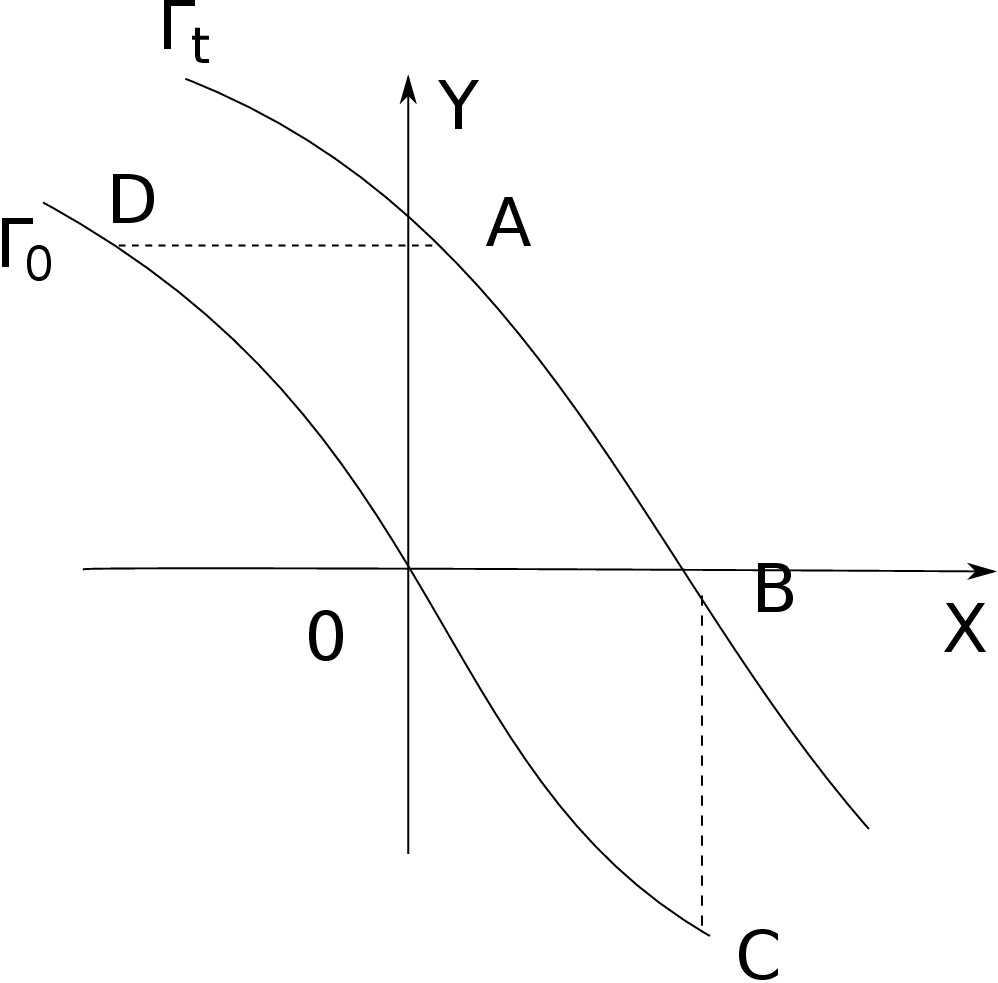}
\caption{{\em Left: The bounded region $\mathcal{D}$ between the two characteristics $x^-$ and $x^+$ and the horizontal line at time $t$. Right: The transformed region $D_t$.}}
\label{fig_last}
\end{figure}
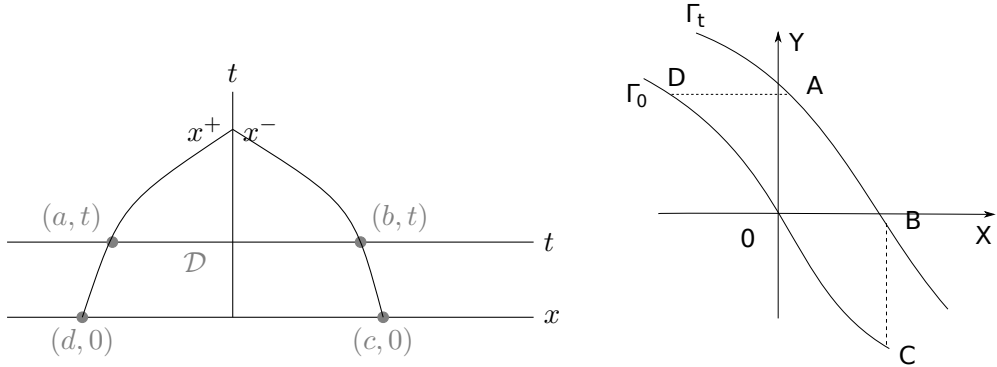

Using $\displaystyle J=\frac{v_t}{g}$ almost everywhere and inequality \eqref{eqn3.3} we have
\beq\label{eEst}
\begin{split}
\int_a^b (\theta_t^2+c^2(\theta)\theta_x^2)(x,t)\,dx \leq& \int_d^c (\theta_t^2+c^2(\theta)\theta_x^2)(x,0)\,dx\\
&-2\iint_{\mathcal{D}}\theta^2_t\,dxdt-2\iint_{\mathcal{D}}\frac{v_t}{g}h\theta_t\,dxdt.
\end{split}
\eeq

We have shown that $J=\frac{v_t}{g}=v_{xx}+\frac{h}{g}\theta_t$ holds true in the $L^2(\Omega_T)$ sense. Thus,
\beq\label{eqn3.4}
\begin{split}
\iint_{\mathcal{D}}\frac{v_t}{g}h\theta_t\,dxdt=\iint_{\mathcal{D}}(\frac{v_t^2}{g^2}-v_{xx} v_t)\,dxdt,
\end{split}
\eeq
where $v_{xx}$, $\theta_t$, $v_t\in L^2(\mathcal{D})$.
Integrating by parts, the second term becomes
\beq\label{eqn3.5}
\begin{split}
-\iint_{\mathcal{D}}v_{xx}v_t\,dxdt
=&\iint_{\mathcal{D}}v_{x}v_{xt}\,dxdt+\int_{AD}\frac{v_xv_t}{\sqrt{1+c^2}}\,ds-\int_{CB}\frac{v_xv_t}{\sqrt{1+c^2}}\,ds\\
=&\frac12\int_a^b|v_x|^2(x,t)\,dx-\frac12\int_d^c|v_x|^2(x,0)\,dx\\
&\quad-\int_{0}^t(v_tv_x)(x^+(t),t)\,dt-\int_{0}^t (v_tv_x)(x^-(t),t)\,dt,
\end{split}
\eeq
where $x^+(t)$ and $x^-(t)$ are characteristic $DA$ and $CB$ respectively.
Substitute this identity into \eqref{eqn3.4} to get 
\beq\label{eqn3.6}
\begin{split}
\iint_{\mathcal{D}}\frac{v_t}{g}h\theta_t\,dxdt=&\iint_{\mathcal{D}}\frac{v_t^2}{g^2}\,dxdt +\frac12\int_a^b u^2(x,t)\,dx-\frac12\int_d^c u^2(x,0)\,dx\\
&-\int_{0}^t(v_tv_x)(x^+(t),t)\,dt-\int_{0}^t (v_tv_x)(x^-(t),t)\,dt.
\end{split}
\eeq
Because $v_t$ is uniformly bounded and $v_x=u\in L^2\cap L^\infty\cap C^\alpha(\overline\Omega_T)$, 
 \[-\int_{0}^t(v_tv_x)(x^+(t),t)\,dt-\int_{0}^t (v_tv_x)(x^-(t),t)\,dt\to 0\;\mbox{ as }\; (a,b)\rightarrow (-\infty, \infty).\]
 Taking $(a,b)\to (-\infty,\infty)$  (so $(d,c)\to (-\infty,\infty)$ too) in \eqref{eEst} and \eqref{eqn3.6}, we have
\beq\label{eqn3.3-2}
\begin{split}
\mathcal E(t)
\leq \mathcal E(0) -\iint_{\R\times [0,t]}(\frac{v_t^2}{g^2}+\theta^2_t)\,dxdt.
\end{split}
\eeq
This completes the proof. 
\end{proof}

\bigskip

\noindent
{\bf Acknowledgement:} The authors thank the anonymous referee whose comments helped improve the paper.
G. Chen and M. Sofiani's research is partially supported by NSF grant DMS-2008504 and DMS-2306258. W. Liu's research is partially supported by Simons Foundation 
Mathematics and Physical Sciences-Collaboration Grants for Mathematicians \#581822.


\begin{thebibliography}{88}

\bibitem{BC2015} A. Bressan and G. Chen, 
Lipschitz metric for a class of nonlinear wave equations.
{\em Arch. Ration. Mech. Anal.} {\bf 226} (2017), no. 3, 1303-1343.

\bibitem{BC} A. Bressan and G. Chen,  
Generic regularity of conservative solutions to a nonlinear wave equation.  
{\em Ann. I. H. Poincar\'{e}--AN} {\bf  34} (2017),  no. 2, 335-354.

\bibitem{BCZ} A. Bressan, G. Chen, and Q. Zhang,
Unique conservative solutions to a variational wave equation. 
{\em Arch. Ration. Mech. Anal.} {\bf 217} (2015), no. 3, 1069-1101.

\bibitem{BH} A.~Bressan and T.~Huang
Representation of dissipative solutions to a nonlinear variational wave equation. 
{\em Comm. Math. Sci.} {\bf 14} (2016), 31-53.

\bibitem{BHY}
A.~Bressan, T.~Huang, and F.~Yu,
Structurally stable singularities for a nonlinear wave equation, 
{\em Bull. Inst. Math.,  Acad. Sin. (N.S.)}
{\bf 10} (2015), no. 4, 449-478. 


\bibitem{BZ} A.  Bressan and Y. Zheng,  
Conservative solutions to a nonlinear variational wave equation.
{\em Comm. Math. Phys.} {\bf 266} (2006), 471-497.

\bibitem{HB} H. Brezis, {\em Functional Analysis, Sobolev Spaces and Partial Differential Equations}, Springer, New York, 2011.

\bibitem{CCD} H. Cai, G. Chen, and Y. Du, 
Uniqueness and regularity of conservative solution to a wave system modeling nematic liquid crystal.
{\em J. Math. Pures Appl.} {\bf 9} (2018), no 117, 185-220.

\bibitem{liucalderer00} M.~Calderer and C.~Liu.
  Liquid crystal flow: Dynamic and static configurations.
 {\em SIAM J.  Appl. Math.} {\bf 60} (2000), 1925-1949.

\bibitem{Cha} S. Chanderasekhar, {\em Liquid Crystals}, 2nd Ed. 
Cambridge University Press 1992.

\bibitem{CHL} G. Chen, T. Huang, and C. Liu, 
Finite time singularities for hyperbolic systems.
{\em SIAM J. Math. Anal.} {\bf 47} (2015), no. 1, 758-785.

\bibitem{CHL20} G. Chen, T. Huang, and W. Liu, 
Poiseuille flow of nematic liquid crystals via the full Ericksen-Leslie model. 
{\em Arch. Ration. Mech. Anal.} {\bf 236} (2020), no. 2, 839–891.

\bibitem{CS22} G. Chen and M. Sofiani, 
Singularity formation for the general Poiseuille flow of nematic liquid crystals.. 
{\em Communications on Applied Mathematics and Computation} (2022), 1-18.


\bibitem{CZZ12} G. Chen, P. Zhang, and Y. Zheng, 
Energy Conservative Solutions to a Nonlinear Wave
System of Nematic Liquid Crystals. 
{\em Comm. Pure Appl. Anal.} {\bf 12} (2013), no 3, 1445-1468.

\bibitem{CZ12} G. Chen and Y. Zheng, 
Singularity and existence to a wave system of nematic liquid crystals. 
{\em J. Math. Anal. Appl.} {\bf 398} (2013), 170-188.

\bibitem{DeGP} P. G. De Gennes and J. Prost, 
{\em The Physics of Liquid Crystals}, 2nd Ed.
International Series of Monographs on Physics, {\bf 83}, 
Oxford Science Publications, 1995.


\bibitem{Eri76} J. L. Ericksen, Equilibrium Theory of Liquid Crystals.
{\em Advances in Liquid Crystals} (G. H. Brown, ed.), Vol. 2, 233-298.
 Academic Press, New York, 1976. 

\bibitem{ericksen62} J. L. Ericksen,  Hydrostatic theory of liquid crystals.
{\em Arch. Ration. Mech. Anal.} {\bf 9} (1962), 371-378.

\bibitem{Evans} L.C. Evans, {\em Partial Differential Equations,} AMS Press.



\bibitem{Fri} A. Friedman, {\em Partial Differential Equations of Parabolic Type,} Prentice-Hall, Inc., 1964.

\bibitem{frank58} F. C. Frank, 
I. Liquid Crystals. On the theory of liquid crystals. 
{\em Discussions of the Faraday Society} {\bf 25} (1958), 19-28.


\bibitem{GHZ} R.~T.~Glassey, J.~K.~Hunter, and Y.~Zheng, 
Singularities in a nonlinear variational wave equation.
{\em J. Differential Equations} {\bf 129} (1996), 49-78.

\bibitem{HardtK87} R. Hardt and D. Kinderlehrer, 
Mathematical questions of liquid crystal theory. 
In Theory and applications of liquid crystals (Minneapolis, Minn., 1985).  
{\em IMA Vol. Math. Appl.} {\bf Vol 5}, 151-184, Springer, New York, 1987.


\bibitem{HR} H.~Holden and X.~Raynaud, 
Global semigroup of conservative solutions of the nonlinear variational wave equation.
{\em Arch. Ration. Mech. Anal.} {\bf 201} (2011),  871-964.

\bibitem{hongxin12} M.~C. Hong and Z.~P. Xin,
\newblock Global existence of solutions of the liquid crystal flow for the
  {O}seen-{F}rank model in {$\mathbb R^2$}.
\newblock {\em Adv. Math.} {\bf 231} (1012),1364-1400.


\bibitem{huanglinwang14} J. R. Huang, F. H. Lin, and C. Y. Wang, 
Regularity and existence of global solutions to the Ericksen-Leslie system in $\mathbb R^2$. 
{\em Comm. Math. Phys.} {\bf 331} (2014), no. 2, 805-850.

\bibitem{HLLW16} T. Huang, F. H. Lin, C. Liu, and C. Y. Wang,
Finite time singularity of the nematic liquid crystal flow in dimension three. 
{\em Arch. Ration. Mech. Anal.} {\bf 221} (2016), 1223-1254.


\bibitem{jiangluo17} N. Jiang and Y. L Luo, 
On well-posedness of Ericksen-Leslie’s hyperbolic incompressible liquid crystal model. {\em SIAM J. Math. Anal.} {\bf 51} (2019), no. 1, 403-434.



\bibitem{llwwz19}
C.~C. Lai, F.~H. Lin, C.~Y. Wang, J. C. Wei, and Y.~F. Zhou.
\newblock Finite time blow-up for the nematic liquid crystal flow in dimension
  two.
\newblock {\em arXiv:1908.10955}, 2019.

\bibitem{leslie68} F. M. Leslie,  
Some thermal effects in cholesteric liquid crystals. 
{\em Proc. Roy. Soc. A.} {\bf 307} (1968), 359-372.

  \bibitem{Les} F. M. Leslie, 
  Theory of Flow Phenomena in Liquid Crystals.
{\em Advances in Liquid Crystals}, Vol. 4, 1-81.
 Academic Press, New York, 1979. 
 
 \bibitem{LTX16} J. K. Li, E. Titi, and Z. P. Xin, 
 On the uniqueness of weak solutions to the Ericksen-Leslie liquid
crystal model in $\mathbb R^2$. 
{\em Math. Models Methods Appl. Sci.} {\bf  26} (2016), no. 4, 803-822.

\bibitem{Lidaqian} T.T. Li,
Global Classical Solutions   for Quasilinear Hyperbolic Systems,
{\em Research in Applied mathematics 32}, Wiley-Masson, 1994.


\bibitem{liyu} T.T. Li and W. Yu, Boundary Value Problems for Quasilinear Hyperbolic Systems,
{\em Duke University Mathematics Series}, Durham, 1985.
 
\bibitem{lin89} F. H. Lin, 
Nonlinear theory of defects in nematic liquid crystals;  phase transition and phenomena. 
{\em Comm.  Pure  Appl. Math.} {\bf 42} (1989), 789-814.

\bibitem{linlinwang10} F. H. Lin, J. Y. Lin, and C. Y. Wang, 
Liquid crystal flows in two dimensions. 
{\em Arch. Ration. Mech. Anal.} {\bf 197} (2010), 297-336.
 
 \bibitem{linlius01} F. H. Lin and C. Liu, 
 Static and dynamic theories of liquid crystals. 
 {\em J. Partial Differential Equations} {\bf 14} (2001), 289-330.

\bibitem{linwang16} F. H. Lin and C. Y. Wang, 
Global existence of weak solutions of the nematic liquid crystal
flow in dimension three. 
{\em Comm. Pure   Appl. Math.}  {\bf 69} (2016), 1532-1571.
 

\bibitem{linwang10} F. H. Lin and C. Y. Wang, 
On the uniqueness of heat flow of harmonic maps and hydrodynamic
flow of nematic liquid crystals. 
{\em Chin. Ann. Math., Ser. B} {\bf 31} (2010), 921-938.

\bibitem{linwangs14} F. H. Lin and C. Y. Wang,
 Recent developments of analysis for hydrodynamic flow of nematic liquid crystals. 
{\em Philos. Trans. R. Soc. Lond. Ser. A Math. Phys. Eng. Sci.} {\bf 372} (2014), 
no. 2029, 20130361, 18 pp.

\bibitem{OH} H. H., Olsen, H. Holden, The Kolmogorov-Riesz compactness theorem
{\em Expositiones Mathematicae} {\bf 28} (2010) 385-394.

\bibitem{oseen33} C. W. Oseen, The theory of liquid crystals. 
{\em Trans.  Faraday Soc.} {\bf 29} (1933), no. 140, 883-899.

\bibitem{Parodi70} O. Parodi, 
Stress tensor for a nematic liquid crystal. 
{\em J. Phys.} {\bf  31} (1970), 581–584.

\bibitem{MS} M. Sofiani, 
On parabolic partial differential equations with Hölder continuous diffusion coefficients. {\em 	 
10.48550/arXiv.2212.08972} (2022).


\bibitem{wangwang14} M. Wang and W. D. Wang, 
Global existence of weak solution for the 2-D Ericksen-Leslie system. 
{\em Calc. Var. Partial Differential Equations} {\bf 51} (2014), no. 3-4, 915-962.

\bibitem{WZZ13} W. Wang, P. W. Zhang, and Z. F. Zhang, 
Well-posedness of the Ericksen-Leslie system. 
{\em Arch. Ration. Mech. Anal.} {\bf 210} (2013), no. 3, 837-855.

 

\bibitem{ZZ03} P. Zhang and Y. Zheng,  
Weak solutions to a nonlinear variational wave equation. 
{\em Arch. Ration. Mech. Anal.} {\bf 166} (2003), 303--319.

\bibitem{ZZ10} P. Zhang and Y. Zheng, Conservative solutions to 
a  system of variational wave equations of nematic liquid crystals. 
{\em Arch. Ration. Mech. Anal.} {\bf 195} (2010), 701-727.

\bibitem{ZZ11} P. Zhang and Y. Zheng, 
Energy conservative solutions to a one-dimensional full variational wave system. 
{\em Comm. Pure Appl. Math.} {\bf 55} (2012), 582-632.

\end{thebibliography}
\end{document}